\documentclass[12pt]{amsart}

\usepackage{amssymb,amsmath,amsfonts,geometry,graphicx,caption,xcolor,array,hyperref, subfigure}
\usepackage[perc]{overpic}
\usepackage[all]{xy}

\newtheorem{theorem}{Theorem}[section]
\newtheorem{corollary}[theorem]{Corollary}
\newtheorem{proposition}[theorem]{Proposition}
\newtheorem{lemma}[theorem]{Lemma}
\newtheorem*{theorem*}{Theorem}

\theoremstyle{definition}
\newtheorem{definition}{Definition}[section]

\theoremstyle{remark}
\newtheorem*{remark}{Remark}

\DeclareMathOperator{\PSL}{PSL}
\DeclareMathOperator{\Mod}{mod}
\renewcommand{\H}{\mathbb{H}}

\title{Lorenz attractors and the modular surface}

\author{Christian Bonatti}
\address{Laboratoire de Topologie, Dijon}
\email{bonatti@u-bourgogne.fr}

\author{Tali Pinsky}
\address{The Technion, Haifa}
\email{talipi@technion.ac.il}
\thanks{T.P.\ was supported by the Israel Science Foundation (grant No. 51/4051)}

\date{\today}

\begin{document}
\begin{abstract}
We define an extension of the geometric Lorenz model, defined on the three sphere.
This geometric model has an invariant one dimensional trefoil knot, a union of invariant manifolds of the singularities. It is similar to the invariant trefoil knot arising in the classical Lorenz flow near the classical parameters.
We prove that this geometric model is topologically equivalent to the geodesic flow on the modular surface, once compactifying the latter.
\bigskip
\end{abstract}
\maketitle

\section{Introduction}

The initial motivation for this paper is the numerical discovery \cite{pinsky2016TopologyLorenz} of an invariant trefoil knot for the classical Lorenz equations. The equations have certain parameter values called T-parameters, at which there are two heteroclinic orbits connecting the three singular points in the equations.
The heteroclinic connections can be extended into an invariant trefoil knot passing through infinity, shown in Figure~\ref{fig:trefoil_attractor}. 
Naturally, these parameters do not constitute an open set, and at such a parameter there can be no partial hyperbolicity to the attractor.

\begin{figure}[ht]
\includegraphics[width=10cm]{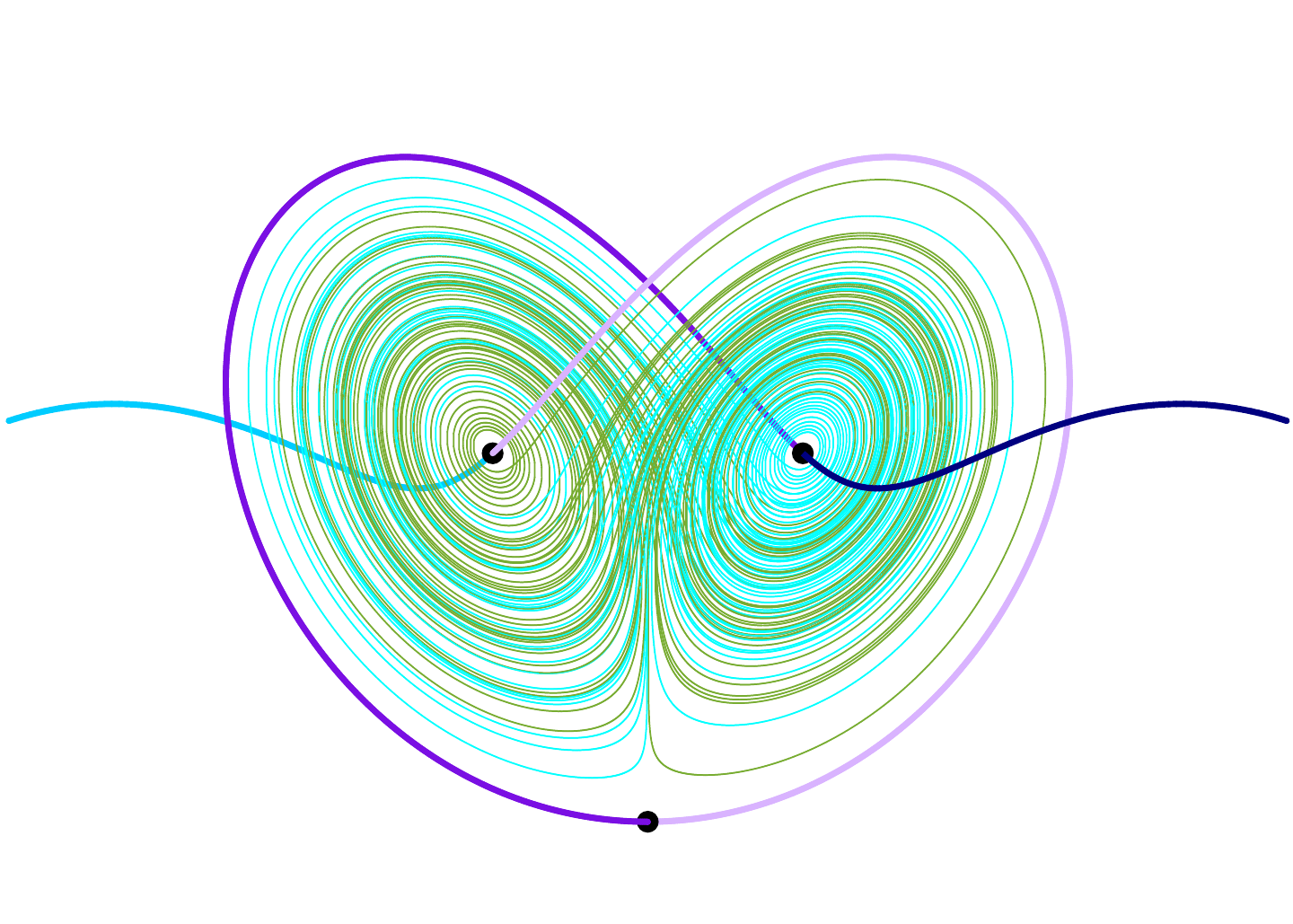}
\caption{The invariant trefoil for the Lorenz system at the T-point ($\rho\approx30.87$, $\sigma\approx10.17$, $\beta\approx\frac{8}{3}$), together with some regular orbits.}
\label{fig:trefoil_attractor}
\end{figure}

We consider the T-parameters in hope to be able to explain in a new way the connection between the Lorenz equations and the geodesic flow on the modular surface found by Ghys \cite{Ghys:KnotsDynamics}: the isotopy classes of periodic orbits of these flows coincide.
The modular geodesic flow is naturally defined on the complement of a trefoil knot in $S^3$. Thus, at the T-parameters, one may aspire to prove that the flows themselves can be compared extending the connection between the periodic orbits.

The modular surface is non-compact, its geodesic flow is defined on the non compact (i.e. cusped) complement of the trefoil knot, and hence it is a structurally unstable dynamical system. 
A way of fixing this problem is by changing the metric, changing the cuspidal end into a funnel and the surface into a infinite volume surface (still of constant curvature -1). The recurrent part of the dynamics now remains in a compact part of the trefoil complement, and a compactification of the trefoil complement containing the recurrent dynamics can be considered. Note that as the resulting system is structurally stable, the above can be done in a canonical way, independent of the change of metric, yielding a flow $\displaystyle\bar\varphi_{\Mod}$ on the compact complement of the trefoil knot.
	
The dynamics of the geodesic flow have a section with a first return map $F$ called the fake horseshoe, that is shown in Figure~\ref{fig:Fake} (see Section \ref{sec:Modular}).
This map has a natural Markov partition consisting of two rectangles, and invariant foliations depicted as vertical and horizontal in the figure.

%

\begin{figure}[ht]
\centering
\begin{overpic}[width=4cm]{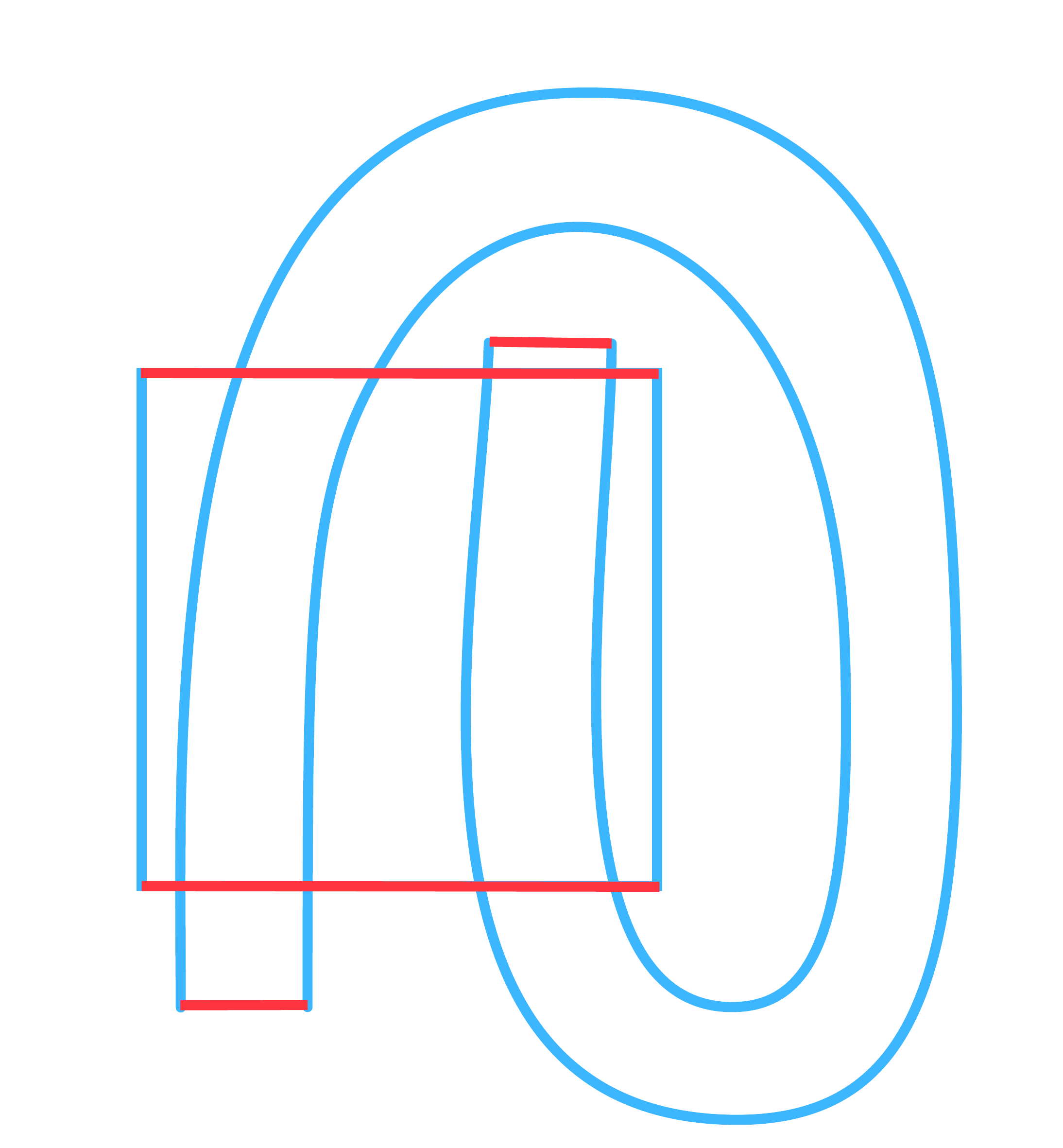}
\put(30,32){$R$}
\put(47, 83){\small{$T(R)$}}
\end{overpic}
\caption{The fake horseshoe is given by a map from a rectangle $R$ to itself, preserving the stable (horizontal) and unstable (vertical) foliations with their orientations, so that the image of each unstable leaf is a union of two unstable leaves. }\label{fig:Fake} 
\end{figure}

\begin{definition}
Given a hyperbolic basic set $K$ of a flow  $X$ on a 3-manifold $M$, a \emph{plug} for $K$ is a three manifold $N$ with boundary,  endowed with a hyperbolic flow $Y$ transverse to the boundary, 
such that a neighborhood of $K$ with the flow $X$ is topologically (orbitally) equivalent by a homeomorphism to a neighborhood of the maximal invariant set of $Y$ in $N$. 

A \emph{model} for $K$ is a plug $(N,Y)$ for $K$ so that the total genus of the boundary of $N$ (i.e. the sum of the genus of all boundary components) is minimal. 
\end{definition}

A main tool in this paper is the following theorem proven in \cite{beguin2002flots},

\begin{theorem*}[Beguin-Bonatti]
Given a hyperbolic basic set $K$ of a flow  $X$ on a 3-manifold $M$
 there exists a unique model $(N,Y)$ for $K$, up to topological equivalence.
\end{theorem*}

Our first theorem provides a way to complete the geodesic flow on the modular surface to a flow on the sphere $S^3$, relating it to the canonical minimal model for the suspension of the fake horseshoe map, which we denote by $\psi_F$.

\begin{theorem}\label{thm:ModularModel}
The flow $\varphi_{\Mod}$ on the complement of the trefoil knot can be extended onto $S^3$ as a structurally stable flow (satisfying Axiom A+strong transversality). Furthermore, the model flow $\psi_F$ for the fake horseshoe is obtained from the resulting flow by removing a small neighborhood of a sink and a small neighborhood of a source.
\end{theorem}

As a consequence, we can prove:

\begin{corollary}\label{cor:FakeModel}
The underlying manifold of the model $\psi_F$ for the fake horseshoe is $S^2\times I$.
\end{corollary}

The Lorenz equations on the complement of the trefoil knot seem to be related to the fake horseshoe map $F$ as well.
When one attempts to establish an equivalence of the geodesic flow with the Lorenz equations at points in which the trefoil exists, one encounters the following issue.
At the T-parameters for the original Lorenz equations at which there is (conjecturally, according to the numerical simulations), a trefoil that is a cycle between a saddle of s-index 2 (the dimension of the stable manifold) 2 and two saddles of s-index 1 with complex unstable eigenvalue.
At such parameters the dynamics is unstable with no hyperbolic or partially hyperbolic structure. 

It is therefore exceedingly possible that any perturbation will create wild dynamics (possessing infinitely many attractors or repellers). 
We present here an approach to turn this flow into a stable one. We suggest that one may separate the singular points from the rest of the non-wandering set by performing a Hopf bifurcation on the two s-index 1 saddles.
Indeed, numerically, such a Hopf bifurcation occurs in the Lorenz family, and can be found at parameters where the trefoil knot connection co-exists with the bifurcation. Thus, our modification to the flow may well lie within the family of flows given by the classical Lorenz equations, for example at the parameter values $(\sigma=10, \rho=20, \beta=\frac{8}{3})$, see \cite{barrio12knead}. However, we do not have much control of the dynamics at these parameters. In particular (if they exist), they lie outside the region shown by Tucker to be equivalent to the geometric model and therefore to be partially hyperbolic.

Thus, instead, we are led to construct a new extension to the geometric Lorenz models, realizing the trefoil connection but at the same time retaining their partial hyperbolicity. This allows us to perform the Hopf bifurcation as above and arrive at an Axiom A system.

An analysis of the first return map to a section for the hyperbolic basic set shows it is again the fake horseshoe $F$.
As the model $\psi_F$ for $F$ is unique, this allows to establish the equivalence between the submanifolds in each of the systems that contain the fake horseshoe. The rest of the dynamics is trivial, allowing us to extend the equivalence to the entire systems.

In Section~\ref{sec:model} we construct the extension of the geometric Lorenz models, designed to model the behaviour of the equations at T-parameters.
This new model is defined on the three sphere $S^3$ and depends on one parameter $s>0$. Similar to the Lorenz equations it has four singular points, one of which is a repeller at infinity. For every $s$, the model  $X_s$ possesses a \emph{heteroclinic knot}, a simple closed curve composed of a union of one dimensional invariant manifolds for the fixed points, which is a trefoil knot, passing through all four singular points.

 Our main result is the following:
\begin{theorem}\label{thm:Equivalence}
The family $X_r$ of flows on the sphere $S^3$ satisfies the following conditions:

\begin{enumerate}

\item For every $r<0$ the flow $X_r$ is singular Axiom A (i.e. the basic sets are either hyperbolic or singular hyperbolic). The chain recurrent set is the union of two sinks, one source and one Lorenz attractor.

\item For each $r>0$, $X_r$ is topologically equivalent to a completion of the geodesic flow  $\varphi_{\Mod}$ to $S^3$.

\end{enumerate}
	
\end{theorem}
 
A more precise statement, and the proof of this theorem is given by Theorem~\ref{thm:construction} and Corollary~\ref{cor:Equivalence}.

%
%
%

A natural question is whether it is possible to establish an orbit equivalence between the Lorenz flow, at a T-point where we have a trefoil connection similarly to our geometric model, and the geodesic flow on the modular surface. As the Lorenz equations cannot currently be shown to be Axiom A on the trefoil complement, this is beyond the scope of the current paper.

\subsection*{Acknowledgments}
This paper originated in discussions during the ``International Conference on Dynamical Systems'' in Buzios in July 2016. We warmly thank the organizers of the conference.

\section{The geometric model}\label{sec:model}
In this section we describe the main object of this paper, a geometric Lorenz model that is a natural extension of the classical models described below, that is defined on the entire three sphere.

The geometric models for the Lorenz equation were introduced in \cite{Guckenheimer1979} and \cite{1977OriginAndStructure} (see also \cite{ThreeDimensionalFlows_book} and \cite{Morales1981}). These models are two parameter families of flows inside a three dimensional region $U$ resembling a butterfly shape, as in Figure~\ref{fig:geometric}. The flows have one singularity $\sigma$ of saddle type, equivalent to the Lorenz singularity at the origin. I.e., the linearization of any such flow has three eigenvalues at the origin, satisfying
$\lambda_3<\lambda_2< 0 <(\lambda_1+\lambda_2)<\lambda_1$.
The flows all have a rectangular cross section $R$, for which the first return map, who's image is the two shaded triangles in the figure. The two parameters determine the location of the tips of the triangles (equivalently, determine the kneading sequence for the unstable manifold of the origin), and the length of the triangles (the amount by which the $x$ direction is stretched). The first return map is discontinuous along a line (the intersection of the stable manifold of the origin with $R$), but is piecewise $C^2$ and  can be proven to be partially hyperbolic. 
The geometric models posses a Lorenz attractor mimicking the one for the original equations around the classical parameters.
The Lorenz equations motivated also extensions  of the classical geometric models \cite{luzzatto2000positive}, modeling the behaviour of the equations where the return map is more complicated.

\begin{figure}[ht]
\centering
\includegraphics[width=9cm]{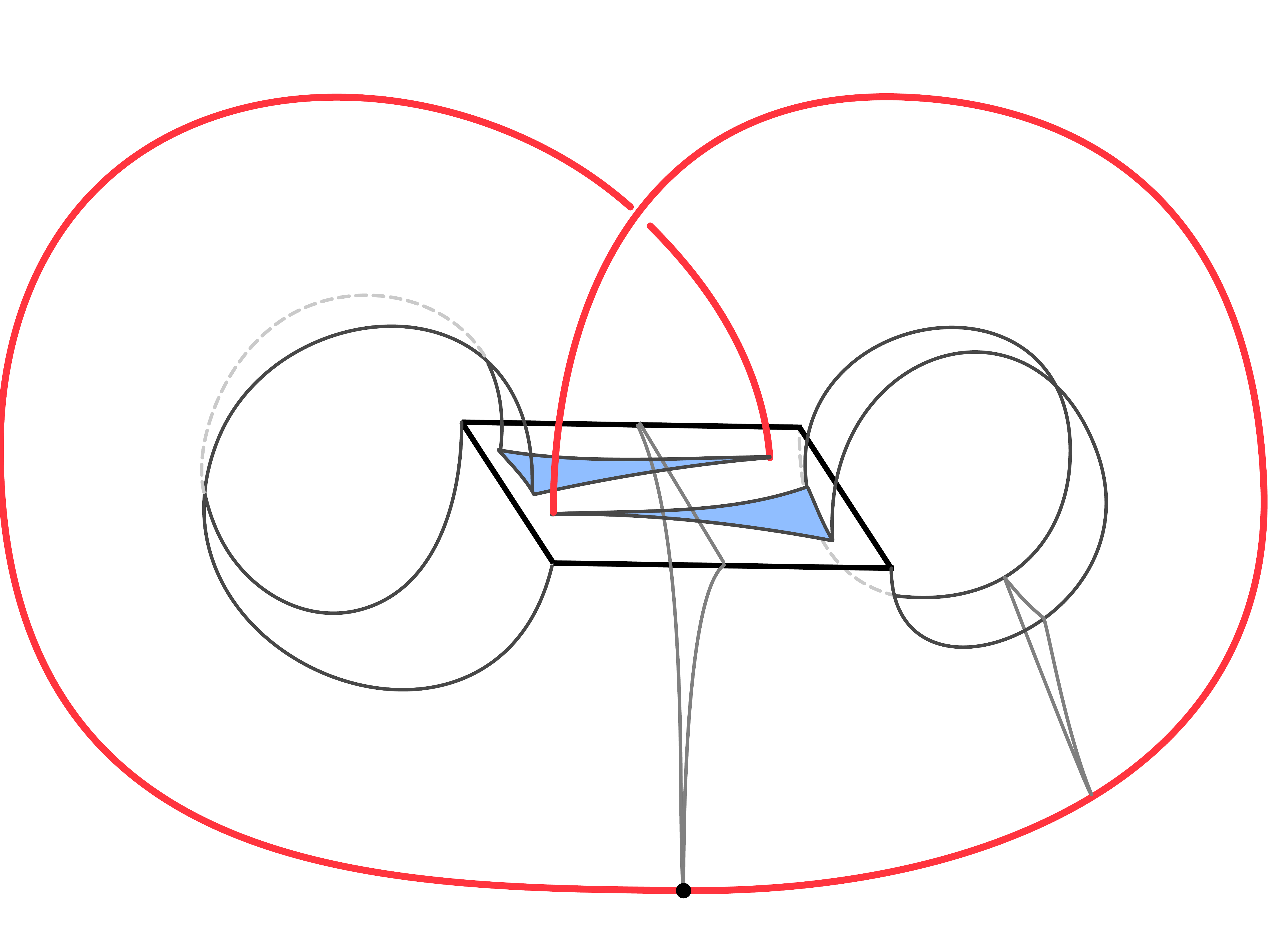}
\caption{The geometric Lorenz model. We may choose coordinates so that the singularity $\sigma$ is at the origin, the Poincare section $R$ has a fixed $z$ coordinate with some $z_R>0$, and the stable folliation of $R$ (induced by the dominated splitting) is parallel to the $y$ axis (the leaves have fixed $x$ coordinate).}\label{fig:geometric}
\end{figure}

The region $U$ on which the classical geometric flows are defined has two regions in its boundary that are homeomorphic to two annuli (the boundaries on the inside of the two butterfly wings). 
Along these annuli, the flow is transverse to $\partial U$ and directed into $U$.
Thus, one can complete the dynamics into a flow on a three ball  $B^3$ by adding two s-index  1 saddle points $p^-$ and $p^+$, one in each wing centre, together with neighbourhoods $N(p^-)$ and $N(p^+)$ around them, so that the stable manifold for $p^\pm$ does not intersect $U$, and the boundary of each neighbourhood is glued to one of the annuli in $\partial U$. We add $p^-$ on the left in Figure~\ref{fig:geometric}, and $p^+$ on the right. The result is a smooth flow on the ball $B^3$, flowing outwards from $N(p^\pm)$ and into $U$ through their joint boundaries.
We put the singularities $p^{\pm}$ to have the same $z$ value (``height'') as the rectangle $R$.

Next, perform a Hopf bifurcation on the saddles $p^-$ and $p^+$ within their neighbourhoods $N(p^-)$ and $N(p^+)$, without modifying the Lorenz attractor.
The effect of such a bifurcation is to transform the saddles $p^\pm$ into sinks $\omega^\pm$ and adding to each of a hyperbolic periodic orbit $\gamma^\pm$ encircling it.

The resulting flow in $B^3$ can then be completed to a flow on $S^3$:
First add a tubular neighbourhood with a flow of the origin and its unstable manifolds, so that the flow enters the neighbourhood. Any flowline continues adjacent to the unstable manifold of the origin, until it enters $B^3$ through $R$.

Next, glue the flow to a 3-dimensional ball which is a neighborhood of a source (``at infinity"), obtaining a flow on $S^3$. 

We next construct a family of flows that continuously passes between the completed geometric model above and a system that possesses a trefoil connection,  so that the dominated splitting holds for any given flow from the family.

\begin{theorem}\label{thm:construction}
There exists a one parameter family of flows $\{X_r\}_{r\in[-1,1]}$,  satisfying the following conditions:
	\begin{itemize}
	\item For any $r<0$, $X_r$ has a chain recurrent set consisting of one Lorenz attractor, one source, two sinks and two periodic saddles (therefore, it is a singular hyperbolic system)
	\item The flow $X_0$ has four chain recurrent classes: one source, two sinks and one singular hyperbolic class containing the periodic saddles $\gamma^-_{0}$ and $\gamma^+_{0}$, and the singularity $\sigma_0$ of Lorenz type.\\
	The invariant manifolds for $\sigma_0$ and $\gamma^-_{0}$,$\gamma^+_{0}$ satisfy: 
	\begin{align*}
	 W^u(\gamma^\pm_{0})\pitchfork W^s(\sigma_0)\ ;\  W^u_+(\sigma_0)\subset W^s(\gamma^+_{0}),  W^u_-(\sigma_0)\subset W^s(\gamma^-_{0}),
	 \end{align*}
	 where $W_\pm^u(\sigma_0)$ are the two components of $W^u(\sigma_0)\setminus\sigma_0$.
	 \item For $r>0$, $X_r$ has five chain recurrence classes: one hyperbolic basic set which is a suspension of the fake horseshoe whose corner orbits are $\gamma^-_{r}$, $\gamma^+_{r}$, one source $\alpha_r$, two sinks $\omega^-_{r}$, $\omega^+_{r}$, and one saddle $\sigma_r$, such that:
	 \begin{align*}
	 W_+^u(\sigma_r)\subset W^s(\omega^+_r), \, W_-^u(\sigma_r)\subset W^s(\omega^-_r)
	 \end{align*}
Furthermore, there exists a $C^1$ embedded simple closed curve, invariant under $X_r$, isotopic to a trefoil knot, consisting of $\alpha_r$, $\omega^-_{r}$, $\omega^+_{r}$, $W^u(\sigma_r)$ and $c_+\cup c_-$ where $c_\pm$ is an arc connecting $\alpha_r$ to $\omega^\pm_{r}$.
	\end{itemize}	
	\end{theorem}
	
\begin{proof}
Fix a flow $X_{-1}$ to be the completion of the classical geometric model as above, starting with a geometric model with any parameters so that the tip of one triangle has the same $y$ coordinate as the base of the other triangle, and the model is symmetric with respect to half a rotation about the $z$ axis.
Perform the Hopf bifurcation to the extent that the periodic saddles $\gamma^\pm$ are contained in $\partial U\cap\partial N(p^\pm)$.

Define the flow $X_s$ to be the one resulting from $X_{-1}$ by pulling the triangle tips and bases along the $x$ direction, while otherwise keeping the structure of the flow and first return map.
In particular, keep the tips and triangles on the same $y$ and $z$ coordinates.
For $z<0$, the tips and bases are contained in $R$, getting closer to the boundary as $s$ increases.
For $z=0$ let the tips and bases intersect $\partial R$.
For $z>0$, push the tips continuously further into the neighbourhoods $N(p^-)$ and $N(p^+)$, leaving the bases on $\partial R$.

\noindent\underline{For $r<0$:} The theorem follows immediately, as the flow is a completion as above for the classical geometric model for some parameters.

\noindent\underline{For $r=0$:} In the deformation, we do not alter the fact that the triangles which are the images of the two sub-rectangles in $R$ intersect near their bases the two sides of $\partial R$ which are part of the stable foliation. Thus, the corner orbits $\gamma^\pm$, which are the unique periodic orbits on these stable leaves, are contained in these two images and are part of the singular hyperbolic chain recurrent class.

The tips of the triangles, that are part of $W^u(\sigma_0)$ are on the stable edges of  $R$ as well. It follows that these edges are part of the stable manifolds of the periodic saddles $W^s(\gamma^\pm_0)$ and thus $W^u_+(\sigma_0)\subset W^s(\gamma^+_{0})$ and  $W^u_-(\sigma_0)\subset W^s(\gamma^-_{0})$.

The unstable manifolds of the saddles, on the other hand, continue parallel to the $x$ axis into $R$. As $R\cap W^s(\sigma_0)$ is exactly the intersection of $R$ with the line $x=0$, the intersection of the invariant manifolds within $R$ are transverse.
As $R$ is a section for the flow through the singular hyperbolic piece, it follows they intersect transversally throughout.

Note that the basin of attraction of this class contains a neighborhood of $\sigma_0$. This basin is not open: it is bounded by the closure of the stable manifolds of the $\gamma_0^\pm$.

The action of the deformation on the return map is shown in Figure~\ref{fig:Section}.

\begin{figure}[ht]
\subfigure[$r<0$]{%
\begin{overpic}[width=7cm]{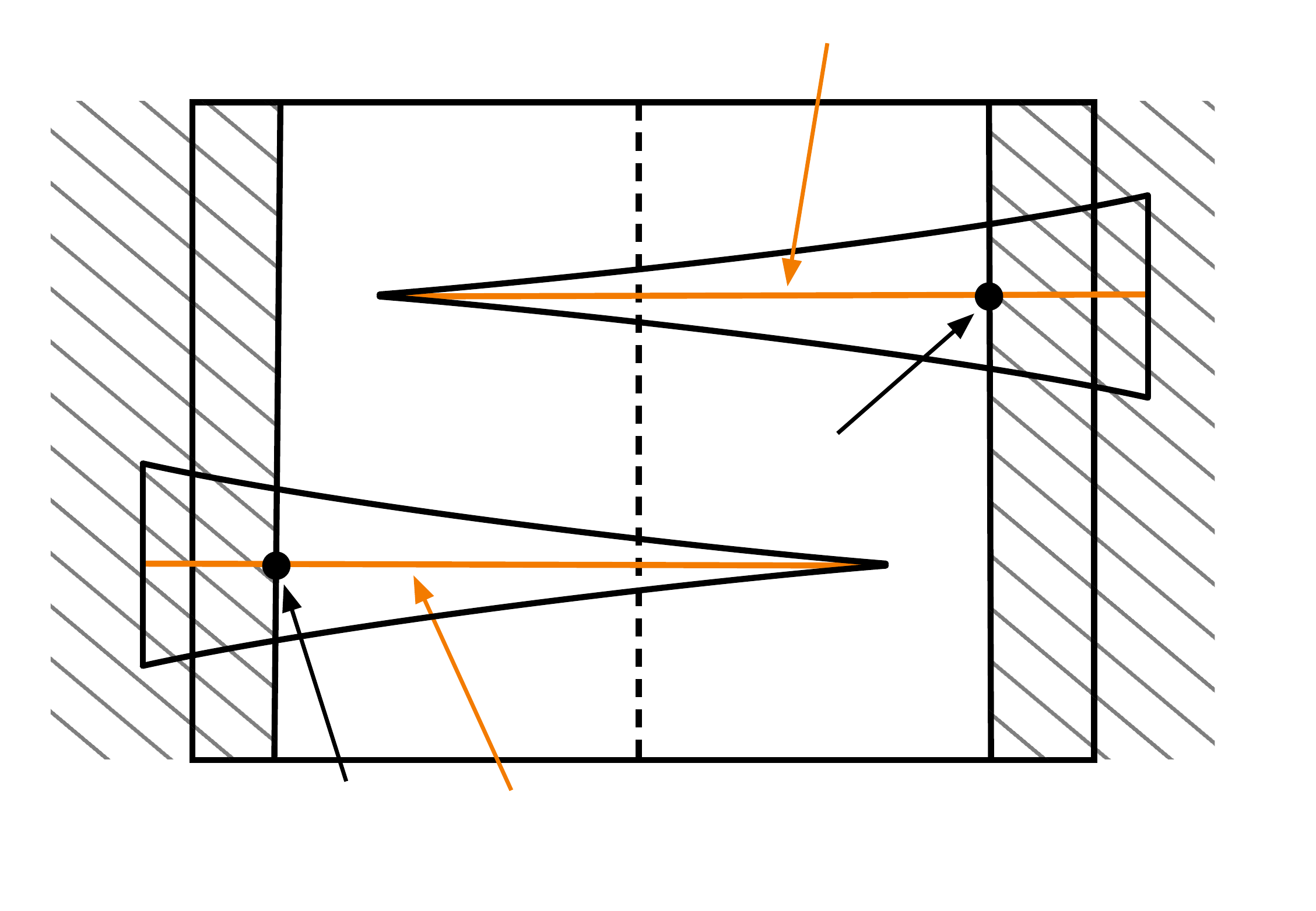}
\put(40,64){$W^s(\sigma_r)$}
\put(-3,45){$W^s(\omega^-_r)$}
\put(85,25){$W^s(\omega^+_r)$}
\put(61,68){$W^u(\gamma_r^+)$}
\put(59,31){$\gamma_r^+$}
\put(34,4){$W^u(\gamma_r^-)$}
\put(22,5){$\gamma_r^-$}
\end{overpic}}
\subfigure[$r=0$]{%
\begin{overpic}[width=7cm]{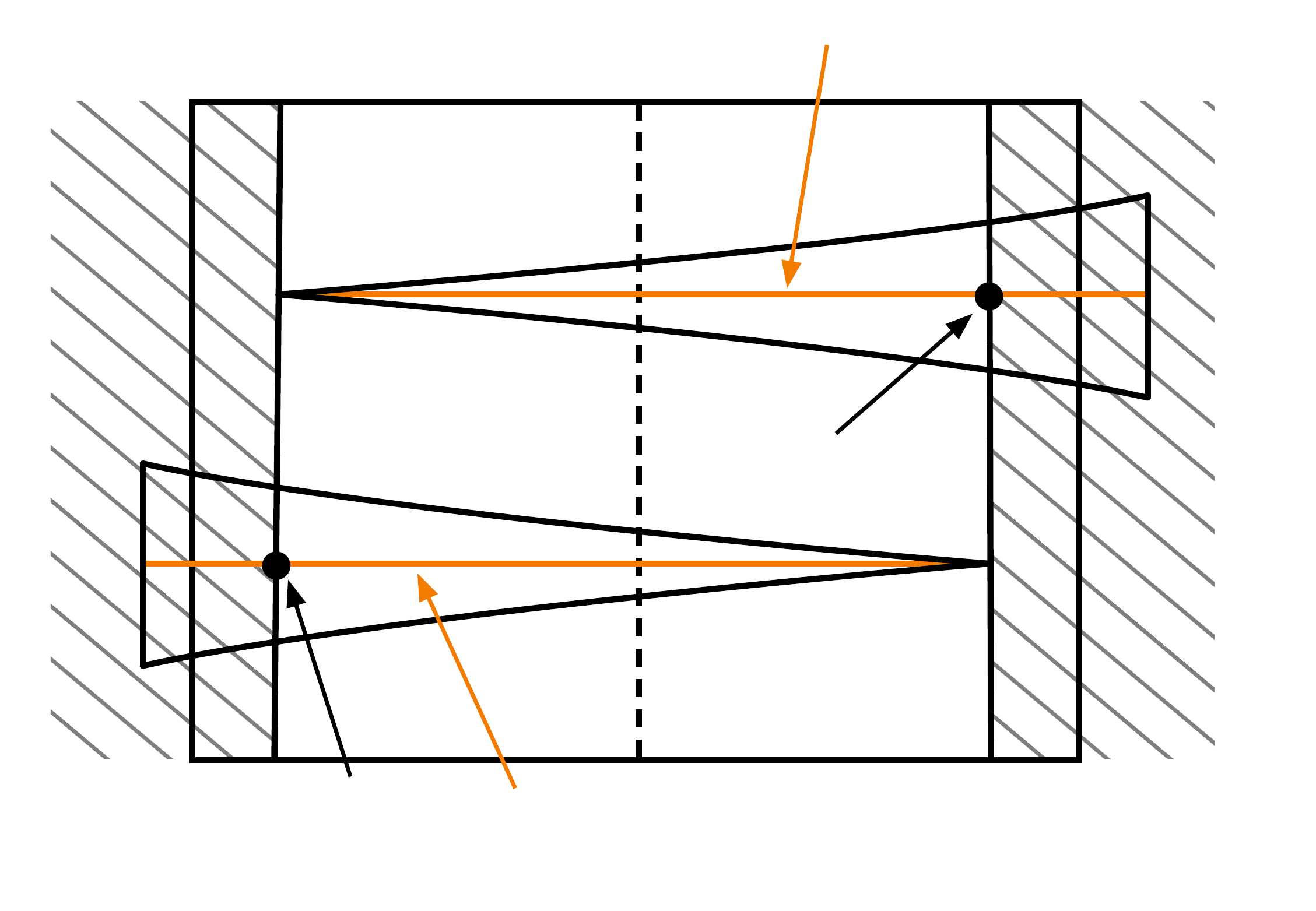}
\put(40,64){$W^s(\sigma_r)$}
\put(-3,45){$W^s(\omega^-_r)$}
\put(85,25){$W^s(\omega^+_r)$}
\put(61,68){$W^u(\gamma_r^+)$}
\put(59,31){$\gamma_r^+$}
\put(34,4){$W^u(\gamma_r^-)$}
\put(22,5){$\gamma_r^-$}\end{overpic}}
\subfigure[$r>0$]{%
\begin{overpic}[width=7cm]{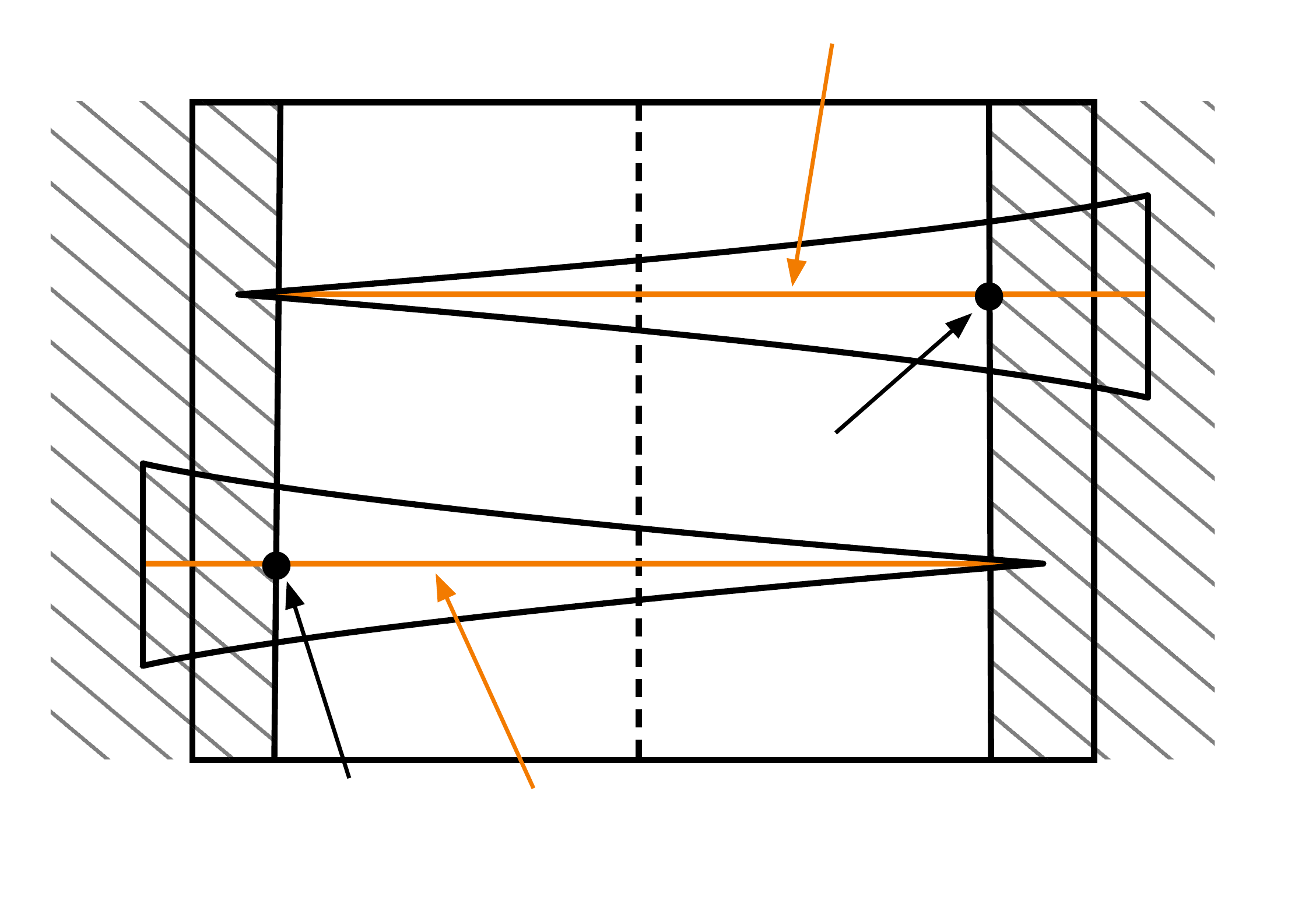}
\put(40,64){$W^s(\sigma_r)$}
\put(-12,50){$W^s(\omega^-_r)$}
\put(85,25){$W^s(\omega^+_r)$}
\put(61,68){$W^u(\gamma_r^+)$}
\put(59,31){$\gamma_r^+$}
\put(34,4){$W^u(\gamma_r^-)$}
\put(22,5){$\gamma_r^-$}
\end{overpic}}
\caption{The return map on the cross-section for $X_r$.}
\label{fig:Section}
\end{figure}
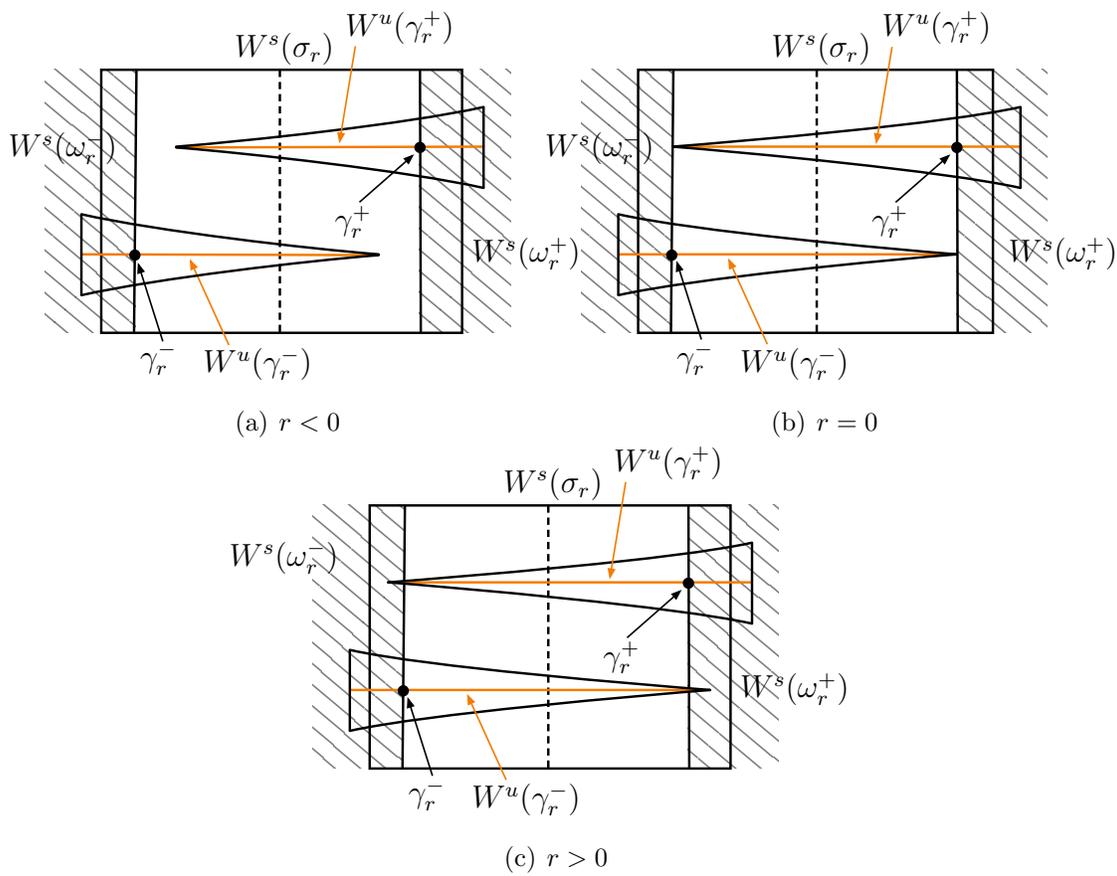

\noindent\underline{For $r>0$:} Here we pull the tips of the triangles further, until they enter the neighbourhoods $N(p^\pm)$, thus they become part of the domain of attraction of the two sinks.
This situation is depicted in Figure~\ref{fig:Geometric2}.
This implies that there exists a small neighbourhood of the saddle $\sigma_r$ so that the $\omega$ limit set of any point in it is either $\sigma_r$ or one of the sinks $\omega^\pm$.
In particular, there is a neighbourhood of the line $x=0$ in $R$ that does not return to $R$. Thus the nonwandering set in $R$ breaks down into a Cantor set. The recurrent set has a geometric type given by the fake horseshoe $F$ (see Figure~\ref{fig:Geometric2}).

\begin{figure}[ht]
\centering
\begin{overpic}[width=12cm]{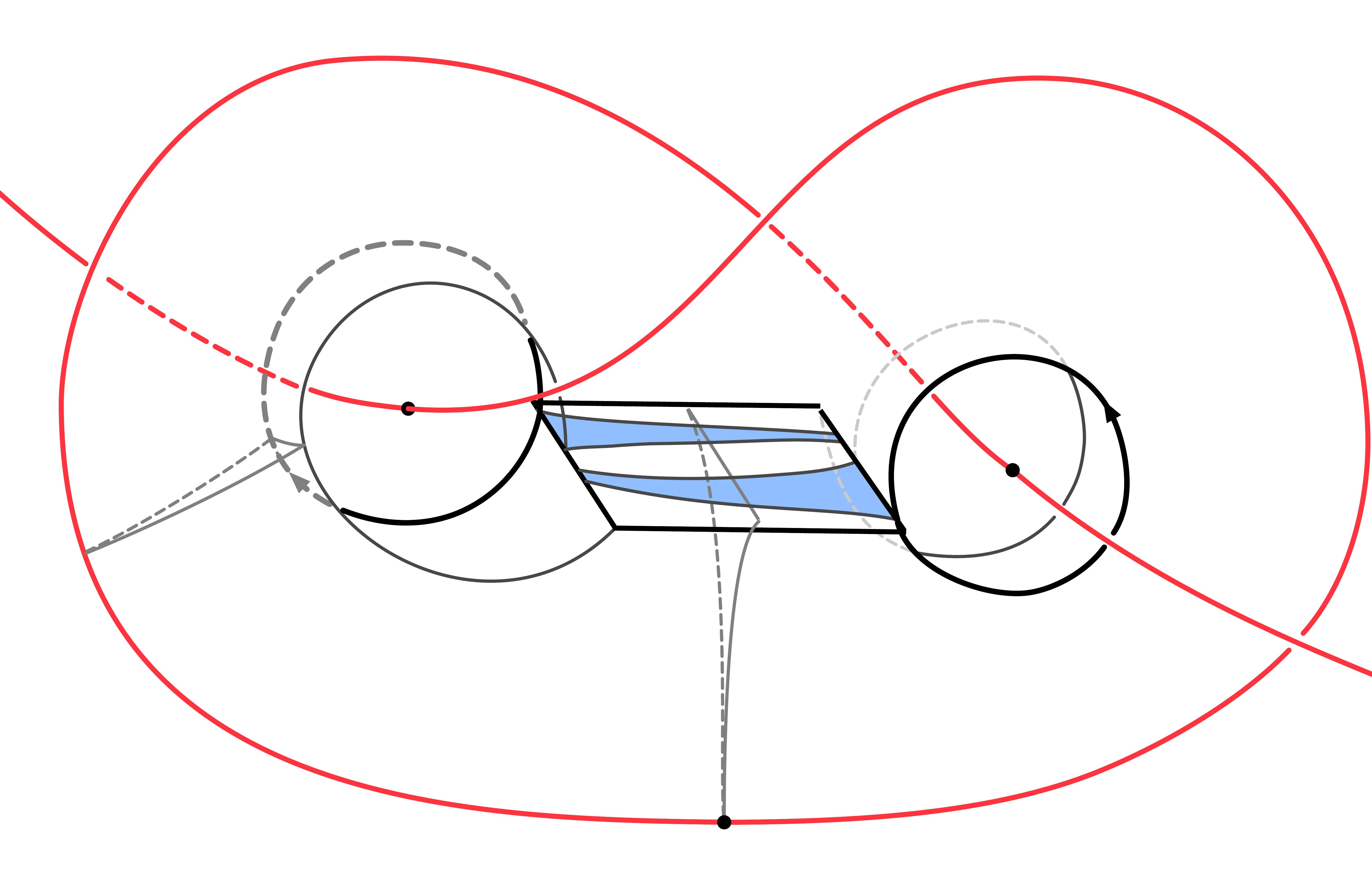}
\put(52,1){$\sigma_r$}
\put(74,31){$\omega_r^+$}
\put(82,35){$\gamma_r^+$}
\end{overpic}
\caption{The new geometric model $X_r$.}\label{fig:Geometric2}
\end{figure}

The fact the bases of the two triangles are on the edges of $R$ implies that $\gamma^\pm$ are the corner orbits for this fake horseshoe.
As we choose the convention that the triangle on the back is pointing to the right, $W^u_+(\sigma_r)$ connects to $\omega^+_r$ from the back of the figure, transversally to the two dimensional manifold $W^u(\gamma^+)$ which within the neighbourhood $N(p^+)$ can be identified with the $y=0$ plane, see Figure~\ref{fig:Cone}. Thus to continue it smoothly we must continue it by a curve on the front side of $W^u(\gamma^+)$. On the other hand, $W^s(\gamma^+)$ is a cone connecting $\gamma^+$ to the source $\alpha_r$. 
Hence, the union of $W^u(\gamma^+)$ and $W^s(\gamma^+)$ together with $\sigma_r$ and $\omega^+_r$ is a two sphere as in Figure~\ref{fig:Cone}. 
The sphere bounds a three ball in which the flow is equivalent to a north to south flow from $\alpha_r$ to $\omega^+_r$. $c_+$ is a flowline of this flow in the ball, and all flowlines are isotopic. Therefore, there is a canonical way to add a curve $c_\pm$ connecting $\alpha_r$ and $\omega^+_r$ so that it is invariant under the flow and connects smoothly to $W^u_+(\sigma_r)$.

\begin{figure}[ht]
\centering
\begin{overpic}[width=7cm]{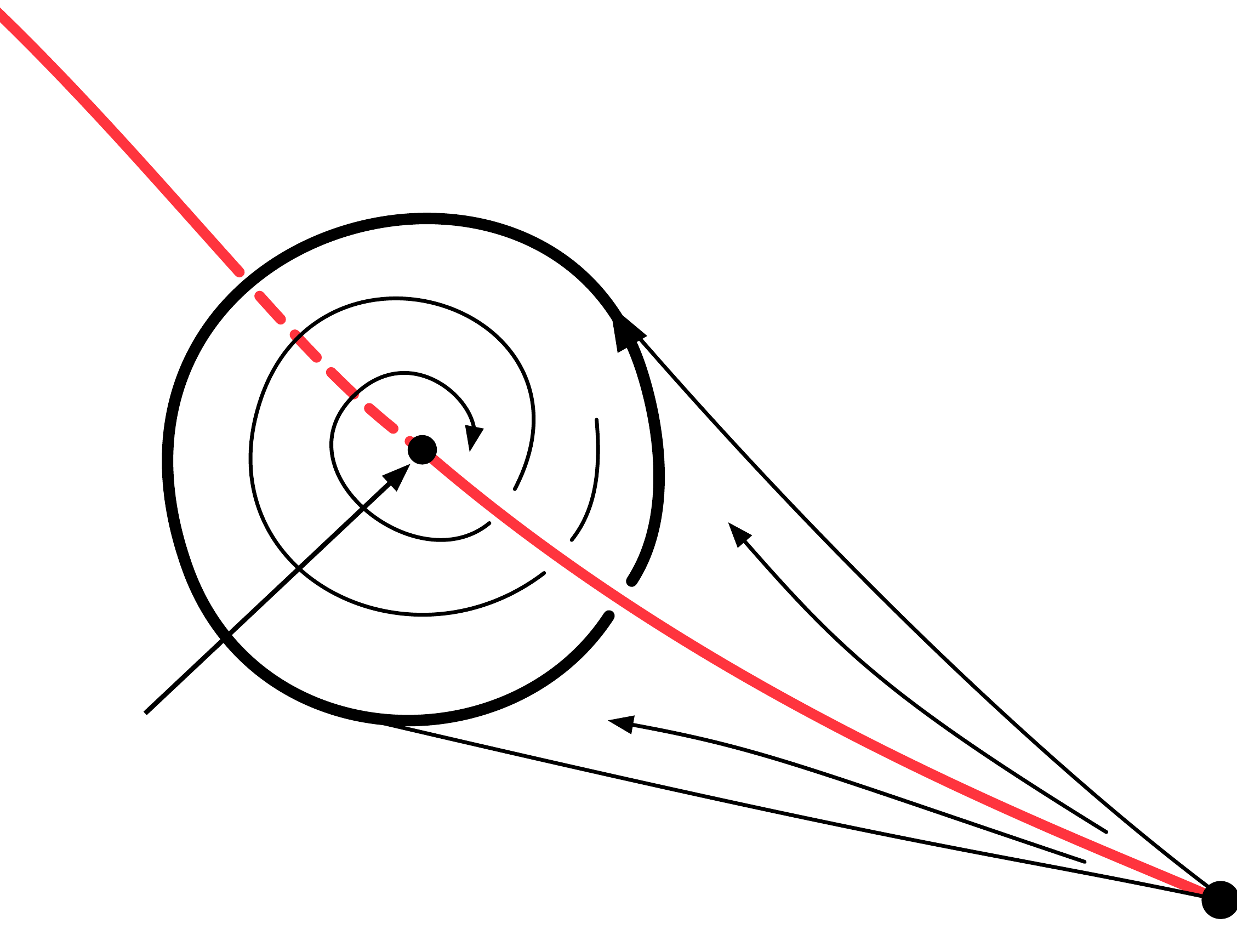}
\put(102,2){$\alpha_r$}
\put(50,54){$\gamma^+_{r}$}
\put(6,14){$\omega_r^+$}
\end{overpic}
\caption{The sphere enclosing the curve $c_+$, which is therefore unique up to isotopy.}\label{fig:Cone}
\end{figure}

Said differently,  $\gamma_+\cup\omega_+$, cuts $W^s(\omega_r^+)$ into two connected components.
One of these components is contained in the unstable basin of $\alpha_r$, the other contains one untable separatrix of $\sigma_r$.  Thus $c_+$ is any orbit in the component which is entirely contained in $W^u(\alpha_r)$.
Similarly, the choice of $c_-$ is a curve connecting $\alpha_r$ to $\omega_r^-$ in the symmetric sphere containing $W^u(\alpha_r)\cap W^s(\omega_r^-)$.

Considering again Figure~\ref{fig:Geometric2}, one finds that $\sigma_r\cup W^u(\sigma_r)\cup\omega^-_r\cup\omega^+_r\cup c_-\cup c_+\cup \alpha_r$ is a trefoil knot.
\end{proof}

\begin{remark}
This construction is not unique because, when $r<0$ the Lorenz attractor depends on two parameters (essentially the itinerary of both unstable separatrices of $\sigma$) whence our family depends on a single parameter.
We could turn it into a canonical construction (up to topological equivalence) by considering a two parameters family.  For such a family, the parameter $r=0$ will correspond to the point of intersection of $W^u(\sigma)$ with both $W^s(\gamma^+)$ and $W^s(\gamma^-)$.
\end{remark}


\subsection{Removing the heteroclinic trefoil}

In this section we show that for $r>0$ we can remove a solid torus from $S^3$ so that $X_r$ becomes a nonsingular Axiom A flow on the trefoil complement.
Let $K_r$, $r>0$, be the trefoil knot invariant under $X_r$ built in Theorem~\ref{thm:construction}.

\begin{proposition}\label{pro:Birkhoff}
For any $r>0$, there is a tubular neighborhood  $\Gamma_r$ of $K_r$, varying continuously with $r$, so that the boundary $\partial \Gamma_r$ is a Birkhoff torus $T_r$:  it consist in 2 annuli each of them
bounded by the two periodic orbits $\gamma_r^+$ and  $\gamma_r^-$, whose interiors are transverse to the vector field $X_r$. $X_r$ points into $\Gamma_r$ along one of these open annuli and points outwards $\Gamma_r$ along the other annulus.  Finally, $S^3 \setminus \text{Interior}( \Gamma_r)$ contains the chain recurrent set of $X_r$ except for the four singular points.
\end{proposition}

\begin{proof}
For each of the orbits $\gamma^\pm$, its unstable manifold is directed towards the singularity it encircles. its stable manifold is tangent to the edge of the rectangular cross section nearest to the same singularity. 

Define the torus $T_r$ as follows: Let $T_r$ be a boundary of a neighbourhood of the trefoil heteroclinic knot for $X_r$, so that $\gamma_r^{\pm}\subset T_r$. Choose $T_r$ so that near the singularity $\omega_r^+$ on the right in Figure \ref{fig:Geometric2} it is wider on the front of $\gamma_r^+$ and narrower at the back, so that it crosses the plane containing the cross section diagonally between the two stable manifolds, as in Figure~\ref{fig:Angle}.

\begin{figure}[ht]
\centering
\begin{overpic}[width=5cm]{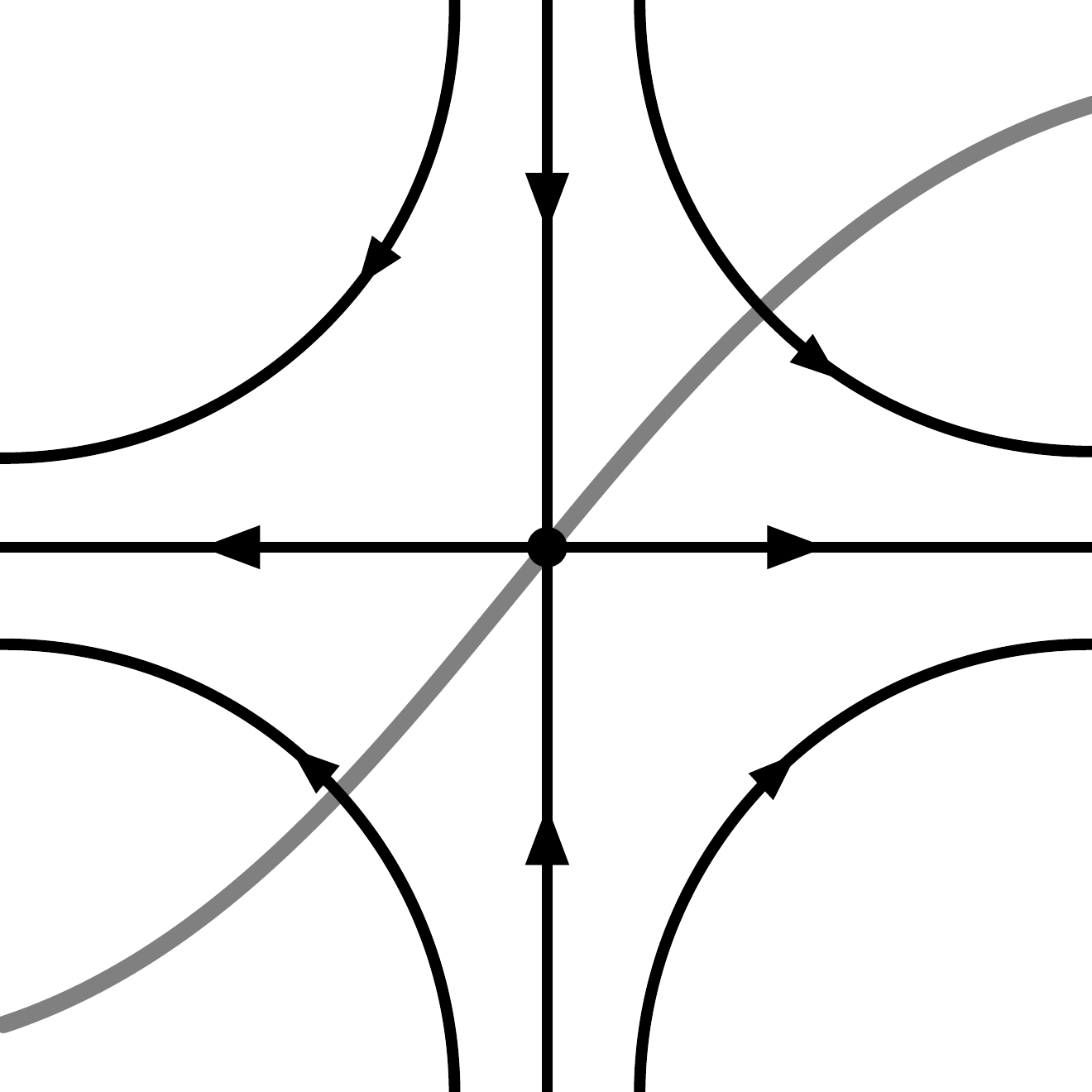}
\put(39,52){$\alpha^+$}
\put(85,75){$T_r$}
\end{overpic}
\caption{A cross section of the orbit $\alpha^+$. The torus $T_r$ passes between the stable and the unstable manifolds of the orbit.}\label{fig:Angle}
\end{figure}

This implies that in a small enough annulus $\mathcal{N}(\gamma_r^+)\cap T_r$ the flow is transverse to $T_r$, except of on its core curve $\gamma_r^+$, flowing into the trefoil neighborhood on the back side of $\gamma_r^+$ and out of the trefoil neighborhood on the front side.

Let $S$ be a neighbourhood of the source $\alpha_r$. In Figure~\ref{fig:Geometric2} the flow is depicted in $\mathbb{R}^3$ by putting $\alpha_r$ at infinity. The flow is transverse to $S$ pointing away from $\alpha_r$. On the front side of Figure~\ref{fig:Geometric2}, one may widen the annulus parallel to the trefoil connection continuing on the front of the figure to infinity. The annulus becomes wider to connect smoothly to the sphere $S$. Note that the outwards flow along this annulus matches the flow through $S$ that also points away from the trefoil (which contains the point at infinity).
Use the symmetry about the $z$ axis to define a similar annulus enclosing the part of the trefoil connecting $\gamma_r^-$ to $\alpha_r$.

Finally, consider a small neighborhood of the saddle $\sigma_r$. An annulus in this neighborhood that encircles the trefoil (parallel to a meridian) will be transverse to the flow, with the flow pointing into the trefoil as the trefoil is tangent to the unstable direction about the origin.
One may extend this annulus, keeping it close enough to the trefoil, until it connects to $A$ while the flow is transverse to the extension and pointing towards the trefoil side of it.
Considering again Figure~\ref{fig:Angle}, one sees the two annuli can be smoothly connected along $\gamma_r^+$, and similarly along $\gamma_r^-$.

Next, consider the solid torus $\Gamma_r$ bounded by $T_r$, depicted in Figure~\ref{fig:Torus}.

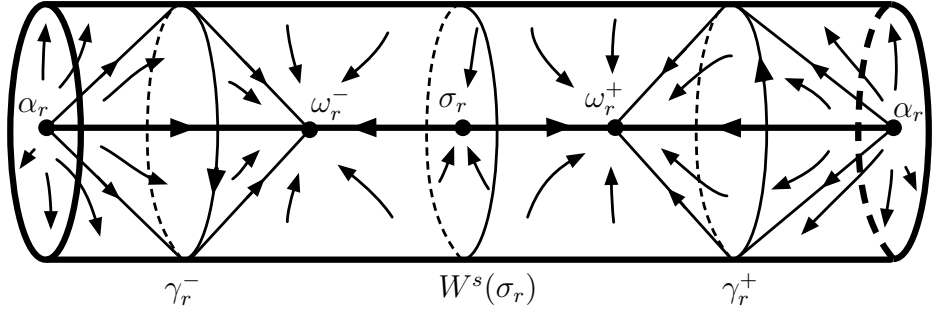
\begin{figure}[ht]
\centering
\begin{overpic}[width=13cm]{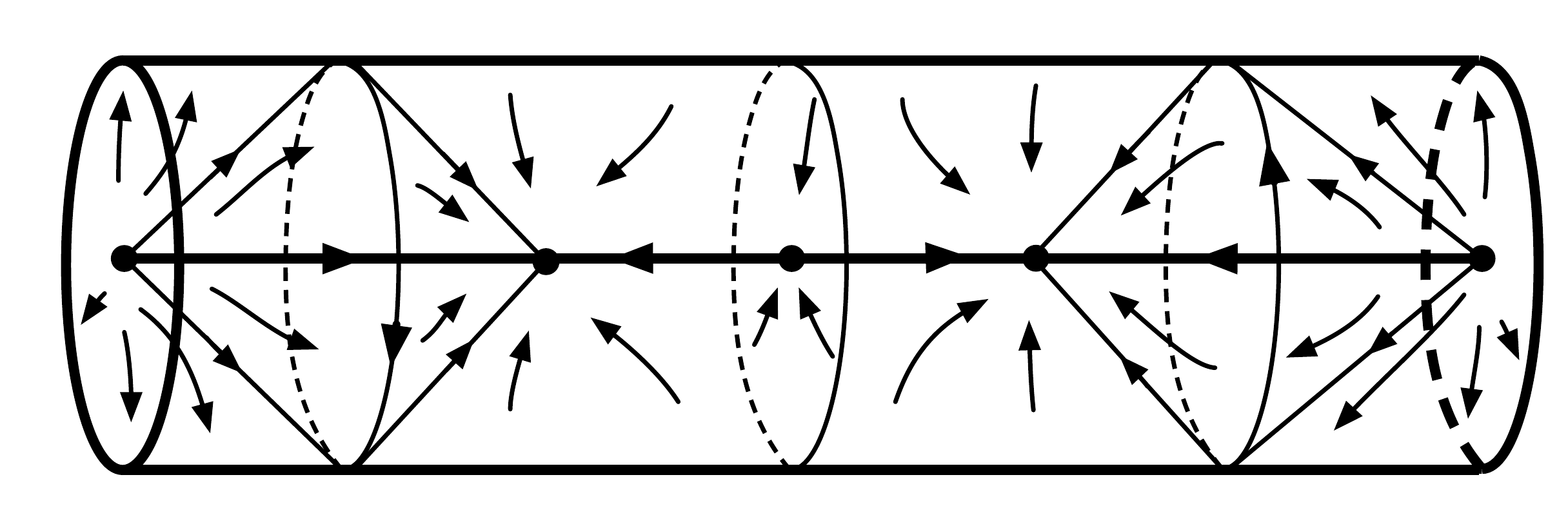}
\put(5,18){$\alpha_r$}
\put(35,18){$\omega_r^-$}
\put(48,18.5){$\sigma_r$}
\put(63,18.5){$\omega_r^+$}
\put(94.5,17.5){$\alpha_r$}
\put(20,-1){$\gamma_r^-$}
\put(48,-1){$W^s(\sigma_r)$}
\put(77,-1){$\gamma_r^+$}
\end{overpic}
\vskip.5cm
\caption{The solid torus $\Gamma_r$ which is the trefoil neighborhood, cut along a disk $D$ containing $\alpha_r$.}\label{fig:Torus}
\end{figure}

$\Gamma_r$ contains the four singular points. The cone like disks in the figure are the invariant manifolds of the tangent orbits $\gamma_r^\pm$. The two dimensional stable manifold of the saddle is also depicted. Together these manifolds divide $\Gamma_r$ into several regions:
\begin{enumerate}

\item The diamond shape regions between the stable and the unstable manifolds of $\gamma_r^\pm$. In this region every point arrives from $\alpha_r$ at $t\rightarrow-\infty$ and limits on $\omega_r^\pm$ at $t\rightarrow\infty$.

\item The region containing $D$, in which every orbit originates from $\alpha_r$ and exits $\Gamma_r$.

\item The region to the right the $\omega_r^-$ cone up to $W^s(\sigma_r)$, in which every orbit emanates on $T_r$ and 
accumulates on $\omega_r^-$ at $t\rightarrow\infty$.

\item The region to the left the $\omega_r^+$ cone up to $W^s(\sigma_r)$, in which every orbit emanates on $T_r$ and 
accumulates on $\omega_r^+$ at $t\rightarrow\infty$.

\end{enumerate}

Note that all orbits contained in the two dimensional invariant manifolds, as well as the connections themselves that are part of the trefoil, are wandering. 
Thus, every orbit in $\Gamma_r$ accumulates onto one of the singular points, as required.
\end{proof}

\begin{definition} Let $A$ be a vector field on $B^3$, so that $A$ is transverse to $\partial{B^3}$ (entering $B^3$), $A$ is Morse Smale, $\Omega(A)$ is composed of 2 sinks and a saddle of s-index 2. The topological equivalence class of $A$ is well defined.
\end{definition} 

We can now prove that the model for the fake horseshoe $F$ is defined on $S^2\times I$:

\begin{proof}[Proof of Corollary \ref{cor:FakeModel}]
Consider again the solid torus $V$, without removing it from $S^3$. Within $V$, there are two transverse spheres for $X^r$, one enclosing a neighbourhood of $\alpha_r$ and one enclosing the three other singularities in $V$.

The flow in one of the three dimensional balls bounded by these two spheres is equivalent to the linearisation about the source. The flow in the other three dimensional ball is equivalent to one given by $A$.
Removing the two 3-balls we arrive at a flow for which the non-wandering set has a section, with $F$ being its geometric type. Denoting this flow by $\psi_F$, it is defined on $S^3\setminus (S^2_0 \cup S^2_1)\cong S^2\times I$.

Thus, $S^2\times I$ is a filtering neighbourhood for $F$.
The two boundaries are transverse to the flow and, since they are both spheres, are clearly of minimal genus. Thus $\psi_F$ on $S^2\times I$ is the canonical model for $F$.
\end{proof}

\section{The modular surface and its geodesic flow}\label{sec:Modular}

The modular surface is the surface $\mathbb{H}^2/\PSL_2(\mathbb{Z})$. It has two cone points and a single cusp. A fundamental domain for this surface is given is depicted in Figure \ref{fig:modular}

\begin{figure}[ht]
\centering
\begin{overpic}[width=5cm]{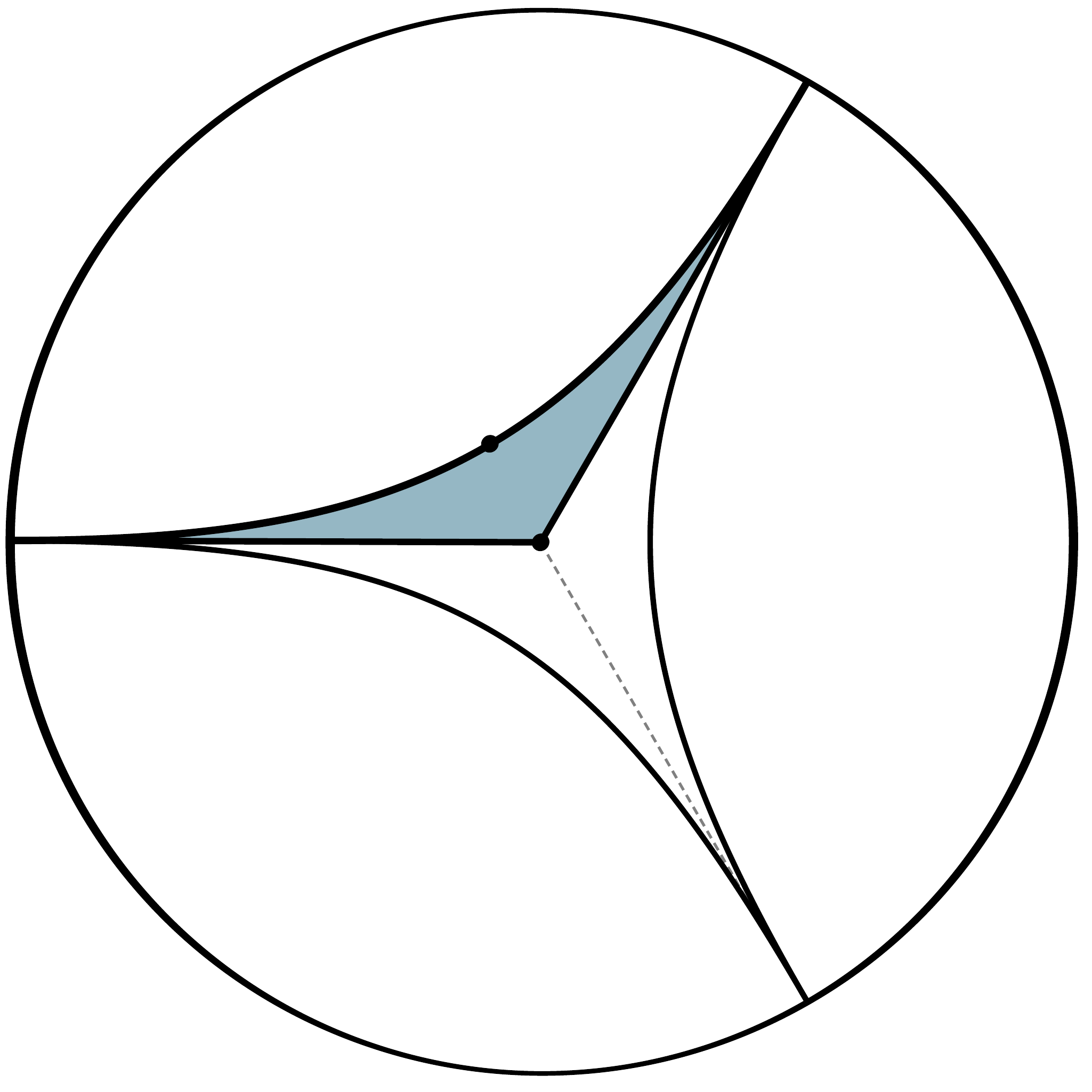}
\put(52,48){$p$}
\put(42,62){$q$}
\end{overpic}
\caption{A fundamental domain for the modular surface.}\label{fig:modular}
\end{figure}

The unit tangent bundle to the modular surface is the space $\PSL_2(\mathbb{R})/\PSL_2(\mathbb{Z})$. It is a Seifert fibered space with two singular fibers corresponding to the two cone points. As a three manifold this space is homeomorphic to the trefoil complement in $S^3$, where the trefoil cusp corresponds to the cusp of the modular surface (see Milnor \cite{Milnor:HyperbolicGeometry}).

On account of the cusp, the geodesic flow on the modular surface $\varphi_{\Mod}$ is not structurally stable,  and it is convenient to switch to a related system, replacing the cusp by a toral boundary. This operation, which we refer to as a compactification is explained in \cite{Ghys:KnotsDynamics} as a deformation  within the representation variety of $\PSL_2(\mathbb{Z})$ in $\PSL_2(\mathbb{R})$. 
First deform the hyperbolic metric on the surface  $\mathbb{H}^2/\PSL_2(\mathbb{Z})$, where the metric is deformed so that  it is still of constant curvature -1 but the cusp turns into a \emph{funnel}, i.e. the surface becomes a hyperbolic surface of infinite area, while keeping the same topology. A fundamental domain for the modular surface with the funnel is shown in Figure~\ref{fig:modular2}.

Note that the element of the fundamental group corresponding to the class of the horocycles in the original metric, contains a unique closed geodesic $h$ in the new metric. The length $l$ of this geodesic defines the new surface $\overline{M}_l$, and the modular surface is the limit $\overline{M}_0$ where $l\rightarrow0$.
The compactification $\varphi^l_{\Mod}$ of the geodesic flow $\varphi_{\Mod}$ on the (compact manifold) $S^3\setminus\text{trefoil}$ is defined as the geodesic flow on the convex core of $\overline{M}_l$. That is, it is the geodesic flow on the surface truncated along the closed geodesic $h$. Note that as any geodesic that crosses $h$ into the funnel never crosses it again, and the same is true in backwards time, the flow $\varphi^l_{\Mod}$ contains the same recurrent dynamics as the geodesic flow on $\mathcal{M}_l$. It follows from their structural stability that the flows $\varphi_{\Mod}^l$ are in the same topological equivalence class for all $l>0$.

The boundary torus $T$ of $S^3\setminus\text{trefoil}$ is the unit tangent bundle to the closed geodesic $h=\partial \overline{M}_l$. Since $h$ is also a closed geodesic, there are two tangent orbits to $T$, $h_0$ and $h_1$, corresponding to traveling along $h$ in both directions.
Since $S^3$ corresponds to the Euler number zero filling of the unit tangent bundle $T^1\overline{M}_l$, see \cite{pinsky2014templates}, the two tangent  orbits are both meridians of the trefoil in $S^3$. These two meridians divide $T$ into two annuli. One of the annuli corresponds to points on the closed geodesic $h$ and tangent direction pointing out of the surface, and the other to points on $h$ and tangent direction pointing into the surface. The flow therefore is transverse to the torus $T$ off the two tangent orbits, i.e. $T$ is a Birkhoff torus, and flows inwards through one annulus and outwards through the other. 

Next consider the recurrent set for the flow. 
The fundamental group for $\overline{M}_l$ is isomorphic to $\PSL_2({\mathbb{Z}})$ and is generated by a rotation of order three around the point $p$ which we depict as the center of the Poincar\'e disk, and a rotation of order two around a point $q$. Let $l_0$ be the hyperbolic geodesic perpendicular to $[p,q]$ at $q$.
Consider a fundamental domain $D$ for $\overline{M}_l$ as given in Figure \ref{fig:modular2} and its three fold cover $D_3$ obtained by the order 3 rotation about $p$. The geodesic segment $l_0$ and its two images $l_1$ and $l_2$, bound $D_3$ together with the three preimages of $h$ intersecting the ends of these geodesics.  The flow $\mathcal{M}_l$ is just the restriction of the geodesic flow on the unit tangent bundle to $\mathbb{H}^2$ to the unit tangent bundle of $D$, with the identifications induced by the action of the fundamental group $\pi_1(\overline{M}_l)$.

\begin{figure}[ht]
\centering
\begin{overpic}[width=7cm]{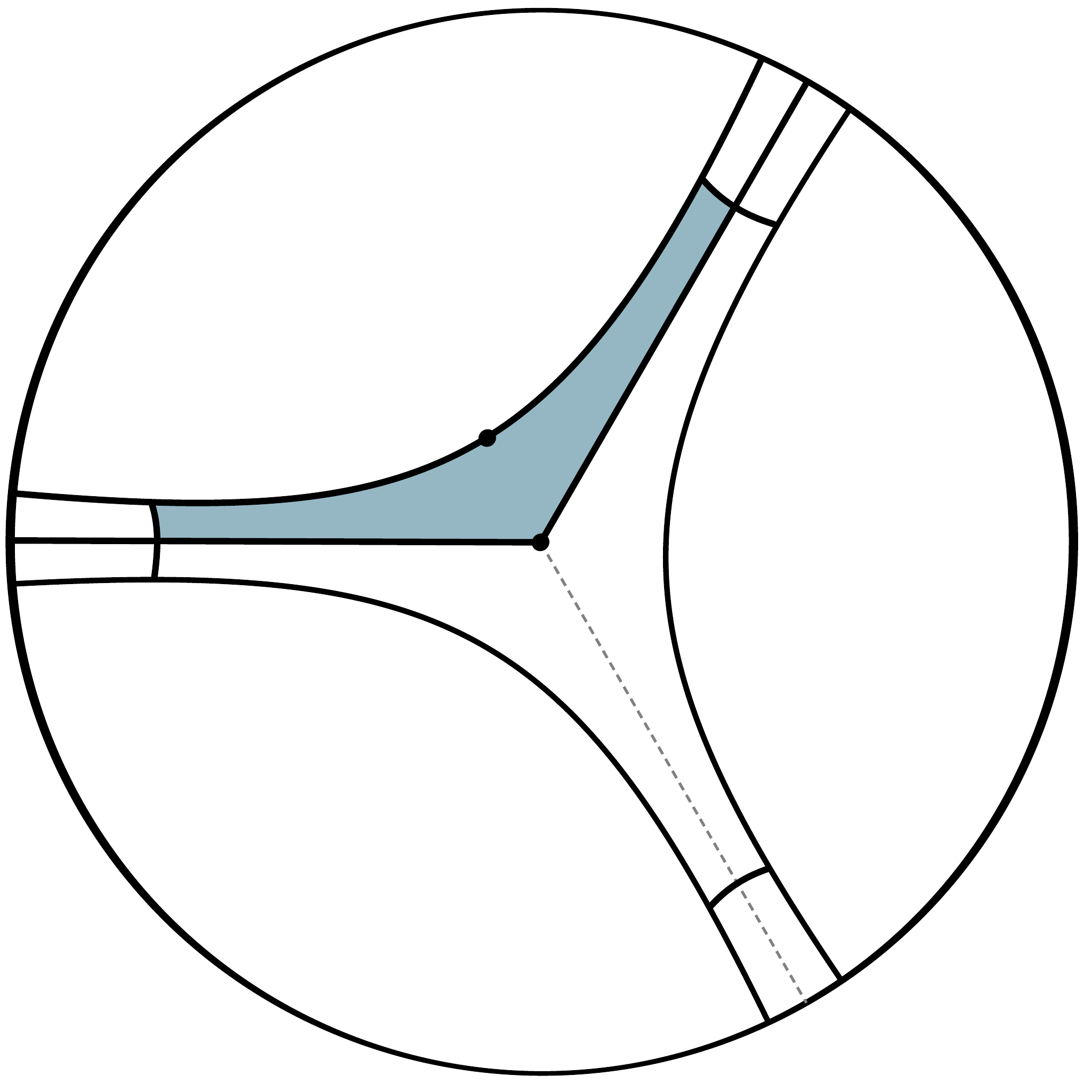}
\put(52,48){$p$}
\put(42,64){$q$}
\put(15,57){$x_0$}
\put(57,82){$x_1$}
\put(15,42){$y_0$}
\put(72,76){$y_1$}
\put(59,15){$z_0$}
\put(71,21){$z_1$}
\put(32,84){$\xymatrix{
        l_0  \ar[dr] &  \\   &  }$}
\end{overpic}
\caption{A fundamental domain $D$ (shaded) for the opened-up modular surface and a three fold cover $D_3$ of $D$ }\label{fig:modular2}
\end{figure}
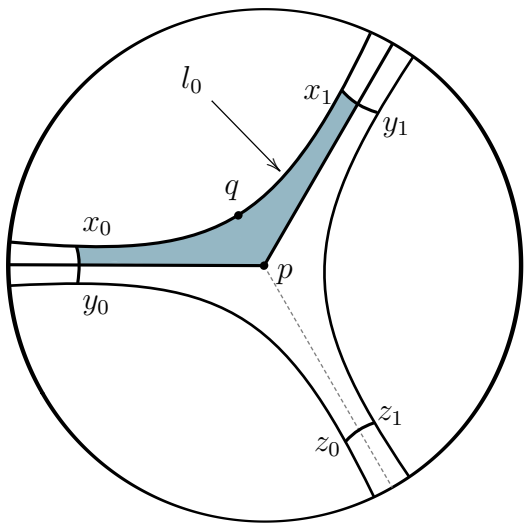

We now proceed to the proof of Theorem \ref{thm:ModularModel}. Some of the results needed in the proof are well known to experts, however we include them here for completeness. 
We obtain the proof by proving two lemmas.

\begin{lemma}[Classical]
The  geodesic flow on $\mathcal{M}_l$ is Axiom A with a basic set the fake horseshoe.
\end{lemma}

\begin{proof}
We begin by constructing a cross section as in \cite{Ghys:KnotsDynamics}.
The geodesic orbits starting at points in the region in $\H^2$ enclosed by $\partial\H^2$ and a lift of the closed geodesic $h$ can only enter the convex core by intersecting $h$. All other orbits there will continue to infinity. It follows that any non-wandering point for the flow has an orbit entirely contained in the convex core.

Therefore, an orbit for a non-wandering point will have a lift crossing the region $D_3$, entering $D_3$ through one of the preimages of $l_0$ and exiting through another such boundary.
The order 3 rotation about $p$ interchanges these three boundary components, and thus we may assume the orbit enters through $l_0$ and exists through one of the other lifts of $l$.
At the point where it exits $T^1(D_3)$, we may use the rotation about $p$ again to make it exit through $l_0$, and then use the rotation about $q$ to obtain an equivalent point that enters $D_3$ again through $l_0$. Thus, the rectangle $R$ consisting of all the points on $l_0$ between its intersections with the lifts of $h$, together with a direction with angle $0\leq\theta\leq\pi$ from the tangent to $l_0$ pointing into the domain $D_3$ is a cross section for the flow on the non-wandering set. 
For a point on $l_0$, we call an angle satisfying $0\leq\theta\leq\pi$ \emph{positive}.

Next, it is a classical fact from hyperbolic geometry, that points on a set of horocycles sharing a point at infinity, with tangent directions directed towards that point (respectively, away from that point) are invariant under the geodesic flow,  and are exponentially and uniformly contracted (expanded resp.). Thus, the stable and unstable foliations are given by the horocycles
\cite{hopf1939}.
For example, if one takes all points on the horocycle $h_c=\{x+Ci\}$ for $\infty$ in the upper half plane model for $\mathbb{H}$, with tangent vectors pointing up to $\infty$, the geodesic flow is a multiplication by $e^t$ which takes $h_c$ to $h_{e^{t}C}$, preserves the upwards direction, and reduces the distances by a factor $t$.

We next find these invariant foliations explicitly within the rectangle $R$, to demonstrate that its geometric type is indeed the fake horseshoe.
Orient $l_0$ from right to left in Figure~\ref{fig:modular2}, and the orient the angles in $R$ counterclockwise from the tangent direction. 
For any point $x$ on $l_0$, there are angles $\theta_0^s(x)<\theta_0^u(x)$ so that $(x, \theta_0^s(x))$ belongs to the stable manifold of $h_0$  and $(x, \theta_0^u(x))$ belongs to the unstable manifold of $h_0$, see Figure~\ref{fig:angles}.  Similarly, for any $x\in l_0$ there are angles $\theta_1^s(x)>\theta_1^u(x)$ for $h_1$ determined by the invariant manifolds of $h_1$ (the geodesic $h_1$ is on the other endpoint of $l_0$, and traverses $l_0$ in the same orientation. Hence the change of sign).

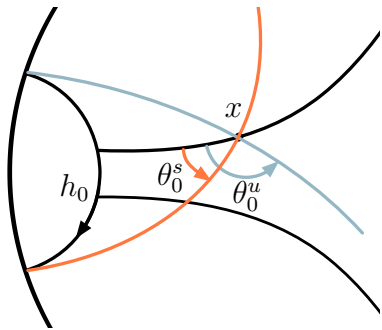
\begin{figure}[ht]
\centering
\begin{overpic}[width=5cm]{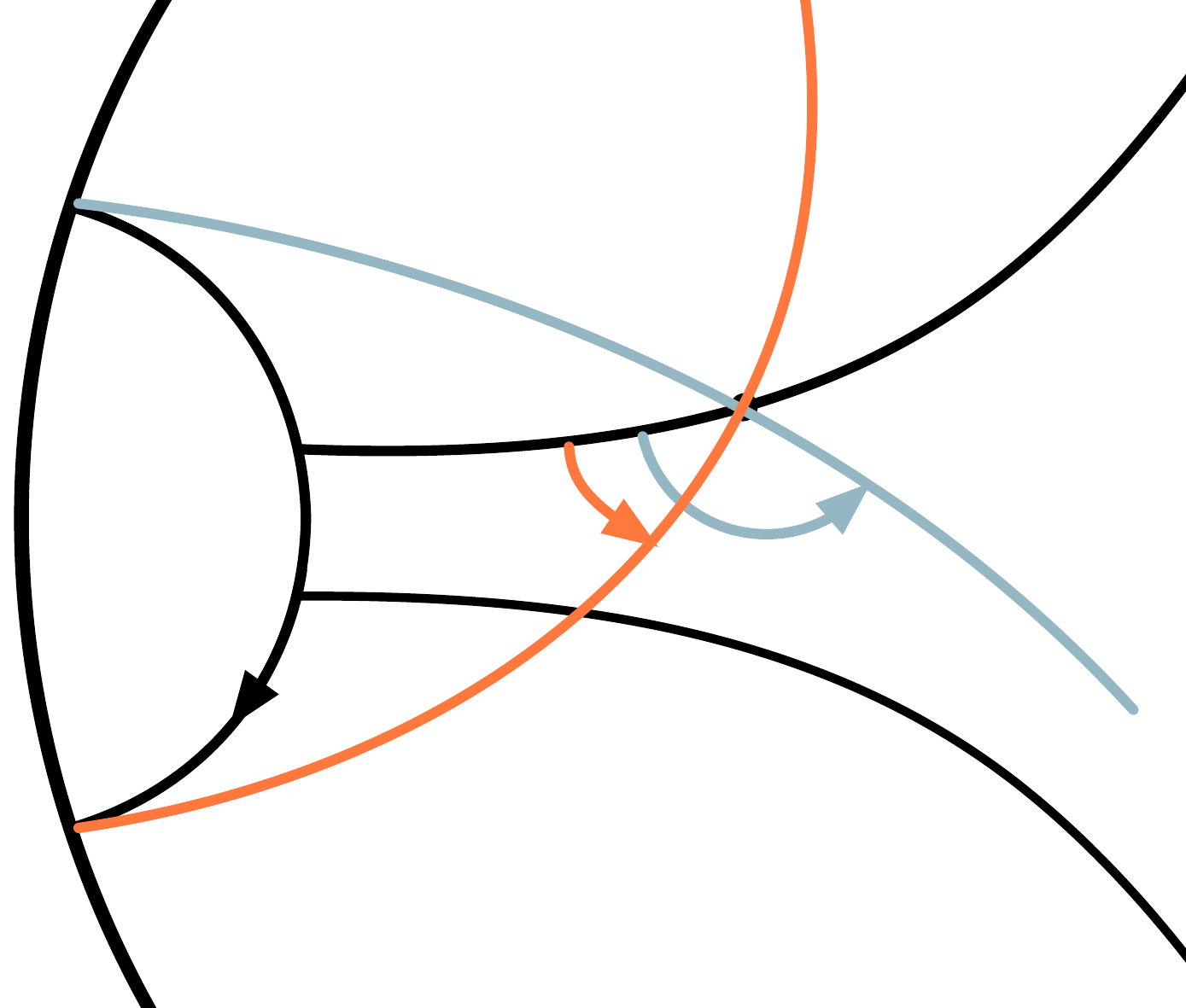}
\put(59,56){$x$}
\put(41,38){$\theta_0^s$}
\put(61,33){$\theta_0^u$}
\put(15,35){$h_0$}
\end{overpic}
\caption{The stable and unstable manifold of each of the corner orbits intersect the fiber above each point on $l_0$.}\label{fig:angles}
\end{figure}

Consider a point $z=(x,\alpha)$ in $R$ with $\alpha<\theta_0^s(x)$. As seen from Figure~\ref{fig:angles}, the orbit of $z$ then crosses the geodesic $h_0$ and remains within the region bounded by it for all forward time. Thus $z$ is a wandering point. Similarly, $z=(x,\alpha)$ for any $\alpha$ satisfying $\alpha>\theta_0^u(x)$, $\alpha< \theta_1^u(x)$ or $\alpha>\theta_1^s(x)$ is wandering.

Therefore we define $R'$ to be the sub-rectangle of $R$ 
$R'=\{(x,\alpha)\,|\,\theta_0^s(x)\leq\alpha\leq\theta_1^s(x),\ \theta_1^u(x)\leq\alpha\leq\theta_0^u(x)\}$.
Define the endpoints of $h_0$ and $h_1$ to be $a_i=lim_{t\rightarrow-\infty}(\phi^t(z)\,)$ for $z\in h_i$, and
$b_i=lim_{t\rightarrow+\infty}(\phi^t(z)\,)$ for $z\in h_i$.
As the horocycles are invariant, it follows that the one dimensional foliations of $R'$:
$\mathcal{W}^s(z)=\{(x,\alpha)\,|\, \lim_{t\rightarrow\infty}\phi^t(x,\alpha)=z$ 
for $z$ between $b_0$ and $a_1$ on the sphere at infinity in the positive direction from $l_0$, 
and 
$\mathcal{W}^u(z)=\{(x,\alpha)\,|\, \lim_{t\rightarrow-\infty}\phi^t(x,\alpha)=z$ for 
$z$ between $b_0$ and $a_1$ in the negative direction,
are invariant. 

Given that the foliations are invariant, we can now compute the first return map.
Consider a single leaf of the unstable foliation $\mathcal{W}^u(z_0)$ and its orbits under the flow as depicted in Figure\ref{fig:Leaf}. The projections to $\mathbb{H}^2$ of forward orbits of the leaf will intersect (at different times) the geodesics $l_1$ and $l_2$, and it is easy to see that any point on $(l_2\cup l_3)\cap D_3$ will be contained in exactly one such projection.
The set of all points obtained in this way above $l_1$: $\mathcal{C}_1=\{(x,\alpha)|\,x\in l_1\}$ is equivalent to a subset $A(\mathcal{C}_1)$ of $R'$, where $A\in \PSL_2({\mathbb{Z}})$ is the element taking $l_1$ to $l_0$ preserving the orientation.
It is easy to see that $A(\mathcal{C}_1)=\mathcal{W}^u(A(z_0))$ as all pointers in $\mathcal{C}_1$ and therefore in $A(\mathcal{C}_1)$ converge to the same point at infinity when $t\rightarrow-\infty$.
In the same way, the set $\mathcal{C}_2$ of all points in the orbits of this leaf above $l_2$ is a subset of $R'$, that is equal to $\mathcal{W}^u(B(z_0))$ for some $B\in \PSL_2({\mathbb{Z}})$.
It is also easy to see no point on an orbit starting at the leaf is equivalent to a point in $R'$ before the orbit reaches the sets $\mathcal{C}_1$ and $\mathcal{C}_2$.

\begin{figure}[ht]
\centering
\begin{overpic}[width=7cm]{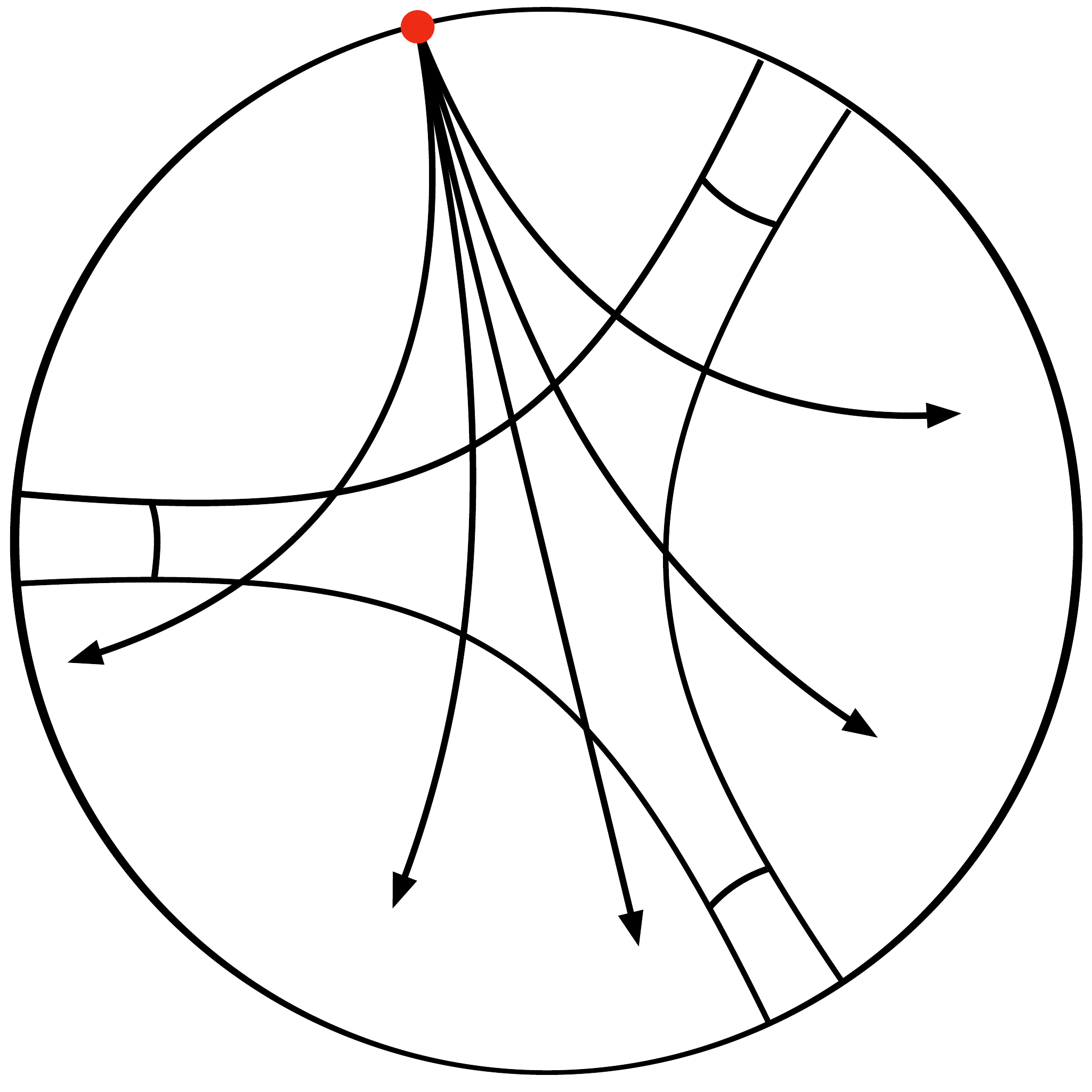}
\put(18,56){$l_0$}
\put(27,41){$l_1$}
\put(68,69){$l_2$}
\put(31,99){$z_0$}
\end{overpic}
\caption{The orbits of the flow of a single unstable leaf  $\mathcal{W}^u(z_0)$.}\label{fig:Leaf}
\end{figure}
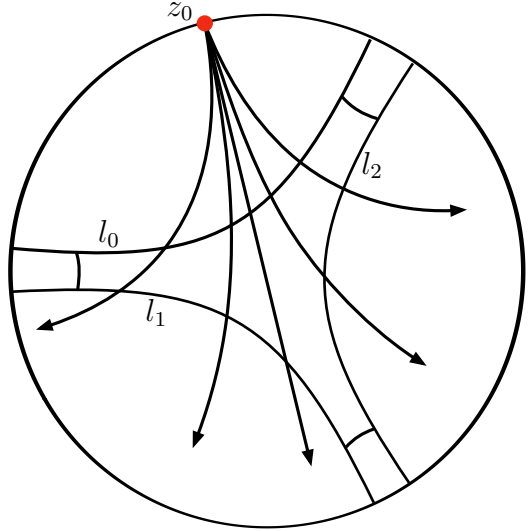

Thus, the first return map for any single unstable leaf is composed of two unstable leaves, and the first return map preserves the orientation the leaves. 
Next consider a single stable leaf $\mathcal{W}^s(w_0)$, the set of all pointers above $l_0$ converging to $w_0$ at infinity. The point at infinity $w_0$ is in the positive direction of $l_0$, and is in the positive direction of either $l_1$ or $l_2$. Thus, the first return map of the leaf will consist of a subset of all pointers on $l_i$ pointing towards $w_0$, which is equivalent to a subset of $l_0$ pointed towards a point $C(w_0)$ for some $C\in\PSL_2(\mathbb{Z})$.  

This suffices to show that the first return map is of the geometric type of the fake horseshoe $F$, as claimed.
\end{proof}

\begin{lemma}\label{lem:Completion}
One may complete the flow $\varphi^l_{\Mod}$ to flow on $S^3$ containing two singular points: a sink and a source.
\end{lemma}

\begin{proof}
We next glue a solid torus $\mathbf{W}$ with a flow $Y$ to the trefoil complement,  so that the the resulting manifold is $S^3$, $Y$ agrees with $\varphi^l_{\Mod}$ on $T=\partial\mathbf{W}$ and thus we can define a flow that is smooth on $S^3$ and agrees with each of the two flows.

The flow on $\mathbf{W}$ is given as in Figure~\ref{fig:Completion}. It contains a source which we denote by $\infty$, and a sink which we denote by $0$, together with their neighbourhoods, and a tube of flow line connecting the source to the sink on either side, so that along each tube we have a tangent orbit. It is clear the only recurrent points in $\mathbf{W}$ are $0$ and $\infty$, and thus the flow is structurally stable.

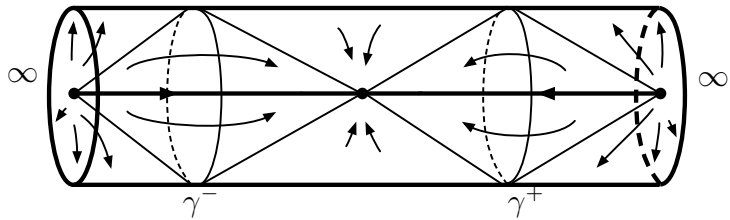
\begin{figure}[ht]
\centering
\begin{overpic}[width=9cm]{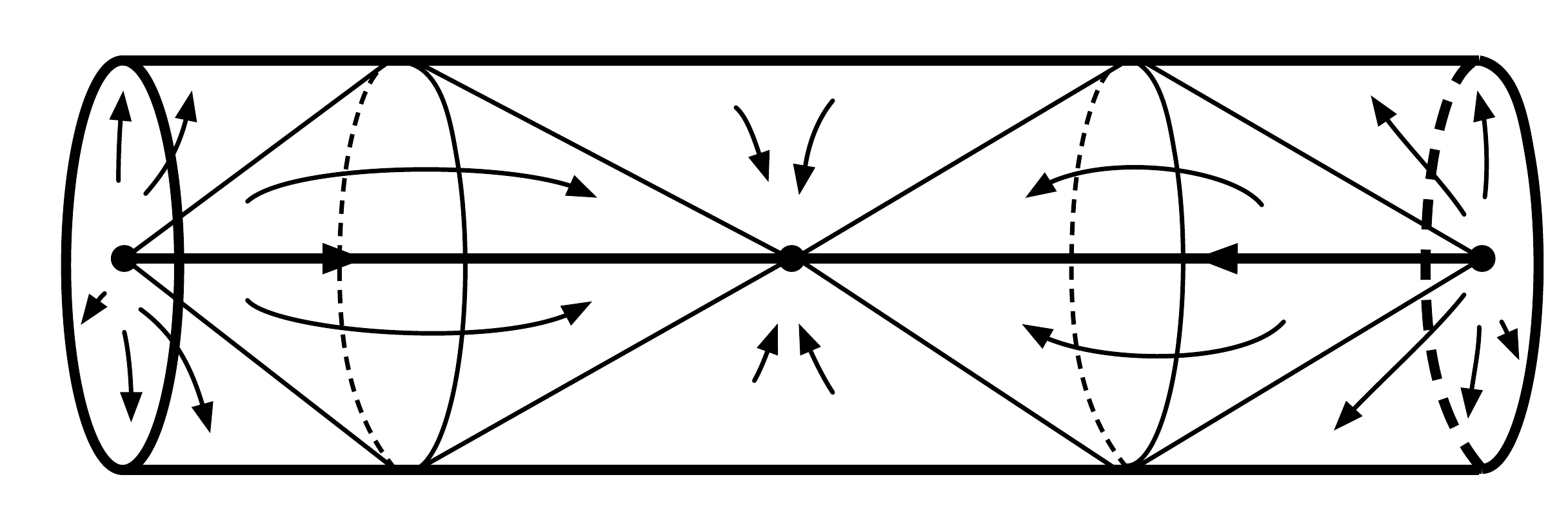}
\put(-2,18){$\infty$}
\put(100,17.5){$\infty$}
\put(24,-1){$\gamma^-$}
\put(72,-1){$\gamma^+$}
\end{overpic}
\vskip.5cm
\caption{The solid torus $\mathbf{W}$ which is glued to the boundary of the trefoil complement to obtain a flow on $S^3$.}\label{fig:Completion}
\end{figure}

As computed in \cite{Ghys:KnotsDynamics}, the slope of the closed geodesic $h$, i.e. the slope of the orbits tangent to the boundary for $\varphi^l_{\Mod}$ is the slope of a meridian for $S^3$. Thus, the resulting manifold is $S^3$.

We can decompose $S^3$ into two three balls and an $S^2\times I$ piece which is the model $\psi_F$ for the fake horseshoe as in Section~\ref{sec:model}.
It follows that $\varphi^l_{\Mod}$ can be completed to $S^3$ as claimed.
\end{proof}

This completes the proof of Theorem~\ref{thm:ModularModel}.
We can now easily prove Theorem~\ref{thm:Equivalence}:

\begin{proof}
The claim for $r<0$ is included in Theorem~\ref{thm:construction}.

For proving the claim for $r>0$, we complete the flow $\varphi^l_{\Mod}$  to a flow on $S^3$ 
As the flow on $\mathbf{W}$ and the flow on the solid torus $\mathbf{V}$ as in the proof of Proposition~\ref{pro:Birkhoff} have the same boundary behaviour and they share the fundamental property that when gluing them to another flow one does not create new recurrent orbits, we may glue the torus $\mathbf{V}$ to $S^3\setminus\text{trefoil}$ instead of 
$\mathbf{W}$ in the proof of Lemma~\ref{lem:Completion}, to obtain a flow $\psi_0$ on $S^3$.

The flows $X_r$ and $\psi_0$ are clearly equivalent on $S^2\times I$ and on each of the two three dimensional balls. The nonwandering set for each of the flows is contained either in  $S^2\times I$ or in one of the balls. This completes the proof.
\end{proof}

We now can complete the proof of Theorem~\ref{thm:Equivalence}, showing the relation of our geometric model to the modular flow:

\begin{corollary}\label{cor:Equivalence}
 $X_r$ for $r>0$ can be defined on the torus complement, and is topologically equivalent to the geodesic flow on the modular surface therein.
\end{corollary}

\bibliographystyle{plain}
\bibliography{bibliography.bib}

\begin{thebibliography}{10}

\bibitem{1977OriginAndStructure}
V.~S. {Afraimovich}, V.~V. {Bykov}, and L.~P. {Shilnikov}.
\newblock {On the origin and structure of the Lorenz attractor}.
\newblock {\em Akademiia Nauk SSSR Doklady}, 234:336--339, May 1977.

\bibitem{ThreeDimensionalFlows_book}
V{\'{\i}}tor Ara{\'u}jo and Maria~Jos{\'e} Pacifico.
\newblock {\em Three-dimensional flows}, volume~53 of {\em Ergebnisse der
  Mathematik und ihrer Grenzgebiete. 3. Folge. A Series of Modern Surveys in
  Mathematics [Results in Mathematics and Related Areas. 3rd Series. A Series
  of Modern Surveys in Mathematics]}.
\newblock Springer, Heidelberg, 2010.
\newblock With a foreword by Marcelo Viana.

\bibitem{barrio12knead}
Roberto Barrio, Andrey Shilnikov, and Leonid Shilnikov.
\newblock Kneadings, symbolic dynamics and painting {L}orenz chaos.
\newblock {\em Internat. J. Bifur. Chaos Appl. Sci. Engrg.}, 22(4):1230016, 24,
  2012.

\bibitem{beguin2002flots}
Fran{\c{c}}ois B{\'e}guin and Christian Bonatti.
\newblock Flots de smale en dimension 3: pr{\'e}sentations finies de voisinages
  invariants d'ensembles selles.
\newblock {\em Topology}, 41(1):119--162, 2002.

\bibitem{Ghys:KnotsDynamics}
{\'E}tienne Ghys.
\newblock Knots and dynamics.
\newblock In {\em International {C}ongress of {M}athematicians. {V}ol. {I}},
  pages 247--277. Eur. Math. Soc., Z\"urich, 2007.

\bibitem{Guckenheimer1979}
John Guckenheimer and Robert~F. Williams.
\newblock Structural stability of lorenz attractors.
\newblock {\em Publications Mathématiques de l'IHÉS}, 50:59--72, 1979.

\bibitem{hopf1939}
Eberhard Hopf.
\newblock {\em Statistik der geod{\"a}tischen Linien in Mannigfaltigkeiten
  negativer Kr{\"u}mmung}.
\newblock 1939.

\bibitem{luzzatto2000positive}
Stefano Luzzatto and Marcelo Viana.
\newblock Positive lyapunov exponents-for lorenz-like families with
  criticalities.
\newblock Societe Mathematique de France, 2000.

\bibitem{Milnor:HyperbolicGeometry}
John Milnor.
\newblock Hyperbolic geometry: the first 150 years.
\newblock {\em Bull. Amer. Math. Soc. (N.S.)}, 6(1):9--24, 1982.

\bibitem{Morales1981}
Carlos~Arnoldo Morales, Maria~Jos\'{e} Pac\'{i}fico, and Enrique~Ramiro Pujals.
\newblock On {$C^1$} robust singular transitive sets for three-dimensional
  flows.
\newblock {\em C. R. Acad. Sci. Paris S\'{e}r. I Math.}, 326(1):81--86, 1998.

\bibitem{pinsky2014templates}
Tali Pinsky.
\newblock Templates for geodesic flows.
\newblock {\em Ergodic Theory and Dynamical Systems}, 34(01):211--235, 2014.

\bibitem{pinsky2016TopologyLorenz}
Tali Pinsky.
\newblock On the topology of the {L}orenz system.
\newblock {\em Proc. A.}, 473(2205), 2017.

\end{thebibliography}

\end{document}